\documentclass[a4paper,10pt]{amsart}
\usepackage{amssymb,tikz-cd,hyperref}
\usepackage[utf8]{inputenc}
\usepackage[english]{babel}
\usepackage{empheq}

\theoremstyle{plain}
\newtheorem{theorem}{Theorem}[section]
\theoremstyle{plain}
\newtheorem{principle}{Principle}[section]

\theoremstyle{definition}
\newtheorem{definition}[theorem]{Definition}
\theoremstyle{plain}
\newtheorem{lemma}[theorem]{Lemma}
\theoremstyle{plain}
\newtheorem{proposition}[theorem]{Proposition}
\theoremstyle{remark}
\newtheorem{remark}[theorem]{Remark}
\theoremstyle{problem}

\theoremstyle{plain}
\newtheorem{corollary}[theorem]{Corollary}
\theoremstyle{definition}

\newcommand{\ddc}{\mathrm{d}\mathrm{d}^\mathrm{c}}

\newcommand{\cn}{\stackrel{\sim}{\rightarrow}}

\DeclareMathOperator{\ch}{ch}
\DeclareMathOperator{\Td}{Td}
\DeclareMathOperator{\RR}{RR}

\DeclareMathOperator{\CH}{CH}

\DeclareMathOperator{\Pic}{Pic}

\newcommand{\PIC}{\mathcal{P}\mathrm{ic}}
\newcommand{\PICm}{\widehat{\PIC}}
\DeclareMathOperator{\Spec}{Spec}

\DeclareMathOperator{\supp}{supp}

\newcommand{\tp}{\bar{\otimes}}
\newcommand{\Vir}{\mathcal{V}\mathrm{ir}}
\newcommand{\Vect}{\mathcal{V}\mathrm{ect}}
\begin{document}
\bibliographystyle{plain}
\title{Deligne-Riemann-Roch Theorems I. Uniqueness of Deligne Pairings and Degree $1$ Part of Deligne-Riemann-Roch Isomorphisms}
\author{Mingchen Xia}
\maketitle

\begin{abstract}
In this paper, we establish a uniqueness theorem for Deligne pairings and prove the degree $1$ part of the Deligne-Riemann-Roch theorem.
\end{abstract}
\setcounter{tocdepth}{1}
\tableofcontents
\section{Introduction}
This is the first of a series of papers. Our final goal is to establish Deligne-Riemann-Roch isomorphisms in various settings.
\subsection{Motivation}
Let $X$, $S$ be arithmetic varieties, namely projective flat scheme over $\Spec \mathbb{Z}$, whose fibre at infinity is regular. Let $f:X\rightarrow S$ be a projective smooth morphism between them. Let $\hat{E}$ be a Hermitian vector bundle on $X$. Let $\hat{\lambda}(\hat{E})$ be the determinant line bundle of $\hat{E}$ along $f$ equipped with Quillen metric.

The celebrated \emph{arithmetic Riemann-Roch theorem} calculates the arithmetic Chern class of $\lambda(E)$ in terms of the arithmetic characteristic classes of $E$ and $T_{X/S}$. See Corollary~\ref{cor:arr} for the precise statement. 
From this piece of information, the isometric class of the line bundle $\hat{\lambda}(\hat{E})$ is determined up to torsion, see~\cite{Sou94lec} Chapter~III, Proposition~1. 

We are then interested in reconstructing $\hat{\lambda}(\hat{E})$ directly from the arithmetic characteristic classes.

We shall use the following notations throughout the paper: $(f:X\rightarrow S)\in \mathcal{F}^n$ means that $f$ is a projective flat morphism of pure relative dimension $n$ between schemes.
\subsection{Deligne pairings}
The main tool that we use
is the so-called \emph{Deligne pairings}. 
This was first proposed in \cite{Del87} by Deligne. In that paper, Deligne provides a conjectural picture of the Deligne pairings. 

To be precise, let $(f:X\rightarrow S)\in \mathcal{F}^n$. Let $L_0,\dots,L_n$ be line bundles on $X$, Deligne conjectured that there exists a functorial choice of a line bundle $\langle L_0,\dots,L_n\rangle_{X/S}$ on $S$, such that
\[
c_1\langle L_0,\dots,L_n\rangle_{X/S}=\int_{X/S}c_1(L_0)\dots c_1(L_n).
\]
Here and in the whole paper, we use interchangeably the notations $\int_{X/S}$ and $f_*$ for the push-forward in Chow group or in some cohomology theory. In complex setting, the smoothness of such operators has been fully studies by Stoll (\cite{Sto66}). We summarize some points in Appendix~\ref{sec:appFib}.
 
Several properties of the pairing are required. (See Definition~\ref{def:DP} for a precise statement.)
Roughly speaking, the requirements are as follows:

1) The pairing should be linear, symmetric and functorial. (See Appendix~\ref{sec:appPic} for the precise meaning.)

2) When $L_0$ can be written as $\mathcal{O}_X(D)$ for a non-zero effective relative Cartier divisor $D$ on $X$, we require that there is a functorial isomorphism
\[
\langle L_0,\dots,L_n \rangle_{X/S}\cn \langle L_1|_{D}\dots,L_n|_{D} \rangle_{D/S}.
\]

3) When $n=0$, we require that the Deligne pairing coincides with the usual norm.

4) When $L_0=f^*L$ for a line bundle on $S$, the pairing should be isomorphic to $L^N$. This number $N$ should be equal to the intersection number of $L_1,\dots,L_n$ on the generic fibre.

Actually, Deligne conjectured that a similar line bundle can be constructed for a general homogeneous polynomial of degree $n+1$ in the Chern classes of vector bundles on $X$.

Deligne also conjectured that there is a functorial isomorphism
\[
\lambda(E)\cn \langle \RR^{n+1}(E)\rangle_{X/S}
\]
for any vector bundle on $E$. 

Deligne actually proved this for $n=1$.

The first successful construction of Deligne pairings is given in \cite{Elk89} by Elkik. Later on \cite{Gar00} relaxed the conditions in Elkik's paper.
There is another approach to Deligne pairings in the literatures, see~\cite{Duc05} for example.

\subsection{Uniqueness of Deligne pairings}
Although the existence of Deligne pairings is well-established in the literatures, there has not been any uniqueness theorems. So we prove one in this paper to fill this gap.

The uniqueness theorem is Theorem~\ref{thm:uniqueness}, it states that under certain natural requirements, the Deligne pairing with the properties stated in the previous subsection is unique up to a unique natural isomorphisms.

A naive idea  is to take divisors of line bundles to reduce the dimension using 2) and induction. 
But this method does not work. The problem is, we require the isomorphisms to be functorial, so they cannot depend on the choice of a specific divisor. 

To overcome this, we actually use an idea of Elkik: by a base extension, we can always assume that a canonical divisor exists. For the details of this point, see Subsection~\ref{subsec:univext}. We shall call this method the method of universal extension.
Roughly speaking, when $L$ is sufficiently ample (Definition~\ref{def:sufamp}), one may construct a projective space $\mathbb{P}$ on $S$, so that the base extension $U:=X\times_S \mathbb{P}$ behaves better. Actually, the line bundle 
\[
\mathcal{L}=L\boxtimes \mathcal{O}_{\mathbb{P}}(1)
\]
admits a canonical section. So we can reduce the definition of
\[
\langle \mathcal{L},L_1,\dots,L_n\rangle
\]
(pull-back maps omitted) to relative dimension $n-1$ if we known that the section is regular.(Namely, the zero scheme $Z(\sigma)$ is flat over $S$.) This is indeed true outside the inverse image of a codimension $2$ set of $\mathbb{P}$. By a descendent result proved below, up to some $\mathcal{O}(1)$ torsion, there is a functorial choice of its descendent.

We still have to deal will the extra factor in $\mathcal{L}$, namely the $\mathcal{O}_{\mathbb{P}}(1)$ torsion, it can be eliminated by 4).

After checking a few compatibility conditions, we conclude the uniqueness of Deligne pairings.

In practice, we find that the universal extension method is more powerful than the uniqueness theorem itself.

\subsection{Deligne-Riemann-Roch theorem}
We are ready to prove the conjecture of Deligne:
\begin{theorem}
Let $(f:X\rightarrow S)\in \mathcal{F}^n$. Assume that $f$ is l.c.i.(locally complete intersection).
Let $E$ be a vector bundle on $X$. Then there is a  functorial isomorphism between functors $\Vect(X)\rightarrow \PIC_{\mathbb{Q}}(S)$
\[
\lambda_{X/S}(E)\cn \langle \RR^{n+1}(E)\rangle_{X/S}.
\]

Moreover, this isomorphism is compatible with short exact sequences. 
\end{theorem}

We shall give a proof that does not rely on the Grothendieck-Riemann-Roch theorem.

One may try to prove this theorem by following the proof of Grothendieck-Riemann-Roch theorem by means of the \emph{d\'{e}vissage} argument of Grothendieck, but we have some difficulities in doing so. Assume that $f$ fits into a commutative diagram as
\[
\begin{tikzcd}
X \arrow[r,"i",hookrightarrow] \arrow[dr,"f"]& \mathbb{P}V \arrow[d,"\pi"] \\ 
&S
\end{tikzcd},
\]
where $V$ is a vector bundle on $S$.

There is no problem for the $\pi$ part. But to deal with the $i$ part, one has to choose a locally free resolution of $i_*E$, say
\[
0 \rightarrow E_N \rightarrow \dots \rightarrow E_1 \rightarrow i_*E\rightarrow 0,
\]
then
\[
\lambda(E)\cong \sum_i (-1)^i \lambda(E_i).
\]

It is not clear that the Deligne pairing part has the same relation, although it should be true by a formal application of Grothendieck-Riemann-Roch theorem. Moreover, we would like to avoid the use of Grothendieck-Riemann-Roch theorem in our proof, so we adopt another approach.

When $E$ is trivial, we have the Koszul resolution of $i_*\mathcal{O}_X$, for which, explicit calculation is easy. So we try to reduce to this case.
Namely, it suffices to prove that the difference $U$ between both sides of the Deligne-Riemann-Roch isomorphism is universal, in the sense that $U$ does not depend on the choice of $L$.

After such a reduction, we then try to make induction on $n$. The main point is that, the determinant line bundles behave well under taking divisors, and hence they admit a large number of universal relations. See Subsection~\ref{subsec:detline} for detail.

In practice, we first resolve the desired isomorphism formally to reduce it to an equivalent condition
\[
\langle \Td^{n+1-i}(T_{X/S})c_1^i(L) \rangle_{X}\cn
i! \sum_{j=0}^{n+1}A^{n+1}_{i,j} \lambda_{X}(jL)
\]
for all $i>0$, where $A^{n+1}_{i,j}$ are constants defined in Appendix~\ref{sec:appCom}.

We then do a formal polarization of both sides to change $c_1^i(L)$ into $c_1(L_0)c_1^{i-1}(L)$. Then we apply the universal extension method to reduce the problem to the case where $L_0$ admits a regular section $D$. Hence we may reduce both sides to relative dimension $(n-1)$ case and apply the already established Deligne-Riemann-Roch theorem to conclude the proof.

During this proof, we find a large amount of combinatorial identities to prove. With the help of Wu Baojun, the author is able to prove all of them. Some necessary combinatorial facts are summarized in Appendix~\ref{sec:appCom}.

Since on $\lambda(E)$, there is a natural metric, namely, the Quillen metric, one wonders if it is possible to reconstruct the metric too. Actually, on the Deligne pairing, there is a natural metric, defined by Elkik. (See Subsection~\ref{subsec:metdel})

We want to calculate their difference. The idea is again to apply the \emph{dévissage} method of Grothendieck. This proof is totally parallel to one part in Gillet-Soul\'{e}'s proof of arithmetic Riemann-Roch theorem. (See~\cite{GS91} for example.)

The result is Theorem~\ref{thm:DRRm}.

This version of Deligne-Riemann-Roch theorem implies the arithmetic Riemann-Roch theorem and the Grothendieck-Riemann-Roch theorem.(See Section~\ref{sec:ap})

Finally, we make a remark on some set theoretic issues in this paper. We encounter frequently objects defined naturally that is unique up to unique isomorphisms. To make notations simpler, we always make a choice of a representative. For this, one need to fix a Grothendieck universe from the very beginning. See Appendix~\ref{sec:appPic}.

We also remark that a similar has been established in an unpublished paper of Franke. See also \cite{Fra90} and \cite{Fra91}. But it seems that at least for the $c_1$ piece, our result is more general than Franke's.
\subsection{Organization of this paper}
In Section~\ref{sec:pre}, we construct the universal extension, prove a descendent lemma of line bundles and establish some algebraic relations of determinant line bundles.

In Section~\ref{sec:gen}, we establish the uniqueness of Deligne pairings and reproduce a construction of Elkik.

In Section~\ref{sec:DRR}, we prove the Deligne-Riemann-Roch theorem.

In Section~\ref{sec:met}, we prove several refined versions of Deligne-Riemann-Roch theorem.

In Section~\ref{sec:ap}, we give a few applications of the Deligne-Riemann-Roch theorem.

In Appendix~\ref{sec:appCom} written with the help of Wu Baojun, we list some combinatorial identities that are used throughout this paper.

In Appendix~\ref{sec:appPic}, we review necessary definitions in the theory of Picard categories.

In Appendix~\ref{sec:appFib}, we give a simplified proof of a theorem of Stoll on the fibre integration.

\subsection{Acknowledgement}
The author would like to show his gratitude to Professor S\'{e}bastien Boucksom for his patient instruction during the last 9 months and for introducing the concept of Deligne pairings to the author.

The author would like to thank Wu Baojun for helping calculate some combiniatonal identities, Dennis Eriksson for a lot of discussions and for correcting a few mistakes in this paper, and Robert Berman for some discussions.

\section{Preliminaries on algebraic geometry}\label{sec:pre}

\subsection{Universal extension}\label{subsec:univext}
Let $(f:X\rightarrow S)\in \mathcal{F}^n$. Namely, $f$ is a projective flat morphism of pure relative dimension $n$ between schemes.

We assume moreover that $S$ is locally noetherian in this subsection.
Let $L$ be a line bundle on $X$. 
\begin{definition}\label{def:sufamp}
We say $L$ is \emph{sufficiently ample} with respect to $f$ if
\begin{enumerate}
\item $L$ is very ample relative to $f$.
\item $R^i f_*L=0$ for all $i>0$.
\end{enumerate}
\end{definition}

Assume that $L$ is sufficiently ample for the time being.

Assume that $E$ is a vector bundle on $S$ such that there is a surjection
\[
g:f^*E\twoheadrightarrow L.
\]

For example, we may take $E$ as $(f_* L)^\vee$, and $g$ is then the canonical morphism $f^* (f_*L)^\vee\twoheadrightarrow L$, which is surjective since $L$ is very ample relative to $f$.

Now we define 
\[
\mathbb{P}:=\mathbb{P}_S(E^\vee).
\] 

We have a Cartesian square
\begin{equation}
\begin{tikzcd}
\mathbb{P}\times_S X \arrow[d,"\pi_1"]\arrow[r,"\pi_2"] & X \arrow[d,"f"] \\
\mathbb{P} \arrow[r,"\pi"] &S
\end{tikzcd}.
\end{equation}

Let $\mathcal{L}$ be the following line bundle on $\mathbb{P}\times_S X$,
\[
\mathcal{L}=\mathcal{O}_{\mathbb{P}}(1)\boxtimes L.
\]

The line bundle $\mathcal{L}$ has a canonical section defined as follows:
A section of $\mathcal{L}$ is a morphism
\[
\mathcal{O}_{\mathbb{P}\times_S X}\rightarrow \pi_1^*\mathcal{O}_{\mathbb{P}}(1)\otimes \pi_2^* L. 
\]

It is then identified with
\[
L^\vee \rightarrow \pi_{2*} \pi_1^*\mathcal{O}_{\mathbb{P}}(1)=f^*\pi_* \mathcal{O}_{\mathbb{P}}(1)=f^* (E^\vee)=(f^*E)^\vee.
\]

Hence, $g$ induces a section of $\mathcal{L}$. We shall denote this section as $\sigma$.

We refer to this construction as universal extension.

\begin{proposition}\label{prop:outcod2}
The section $\sigma$ is regular. 

Moreover, there exists a closed subset $Z\subset \mathbb{P}$ of codimension at least $n+1$, such that
$(\pi_1:Z(\sigma)-\pi_1^{-1}(Z)\rightarrow \mathbb{P}-Z)\in \mathcal{F}^{n-1}$.

Moreover, locally $Z$ is a union of finite projective linear subbundles of $\mathbb{P}$, each having codimension at least $n+1$.
\end{proposition}  
\begin{proof}
That $\sigma$ is regular is shown in \cite{Mun95} 2.2. We prove the other parts.	
	
We may assume without loss of generality that $S$ is irreducible and affine.

Let $\eta_i$ be the generic points of $X$. A section $s\in H^0(X,L)$ is regular if{f} $s(\eta_i)\neq 0$ for any $i$. And $s(\eta_i)=0$ defines line subspaces $K_i$ of $E$ through $g$. 

Notice that $E/K_i\twoheadrightarrow H^0(\bar{\eta_i},L)$ is surjective, as $L$ is ample, we find $\dim E/K_i\geq n+1$.

The proposition now follows from the definition of $\sigma$.
\end{proof}

One may replace $\mathbb{P}$ by the product 
\[
\mathbb{P}_1\times\dots \times\mathbb{P}_d,
\]
each corresponding to a vector bundle $E_i$ on $S$ and a surjection
\[
g_i:f^*E_i\twoheadrightarrow L.
\]
Now we have a regular sequence of sections $(\sigma_1,\dots,\sigma_d)$, where $\sigma_i$ is the universal section of $\mathcal{L}_i=L\boxtimes \mathcal{O}_{\mathbb{P}_i}(1)$. $(\sigma_1,\dots,\sigma_d)$ has fibrewise codimension $d$ outside a codimension $n+2-d$ subset of $\mathbb{P}_1\times\dots \times\mathbb{P}_d$.

The proofs are similar to the cases given above. Readers can also find a proof in \cite{Gar00} Proposition~2.4.1. 

\subsection{Descendent of line bundles}\label{subsec:desclb}
Now we consider the following problem: Let $S$ be a locally noetherian scheme. 

For the descent problem in this subsection, we may work on each connected component of $S$, so we assume from the beginning that $S$ is connected.

Let
\[
\pi:\mathbb{P} \rightarrow S
\]
be a projective space bundle over $S$. Given a line bundle $\mathcal{L}$ on $\mathbb{P}$, we want to choose functorially a line bundle $L$ on $S$ so that
\[
\mathcal{L}\cong L\otimes \mathcal{O}_{\mathbb{P}}(N)
\]
for some $N\in \mathbb{Z}$. 

For the existence of such decomposition and uniqueness of $N$, see~\cite{EGAII} Remarque~4.2.7.

Now it follows from \cite{H} Exercise~III.12.4 (note that no technique assumption there is necessary) that we may take $L=\pi_*(\mathcal{L}\otimes \mathcal{O}_{\mathbb{P}}(-N))$.

\begin{theorem}\label{thm:desc}
The descent procedure induces equivalences between categories
\[
\PIC(\mathbb{P}) \rightarrow \PIC(S)\times \mathbb{Z},
\]
where $\mathbb{Z}$ is the discrete category on the set $\mathbb{Z}$, the second component is given by the $N$ stated above.
\end{theorem}
For the proof, see \cite{EGAIV} Lemme 21.13.2 .

In particular, in order to establish an isomorphism between descendent line bundles, it suffices to do so for the original line bundles.

Now we shall recall briefly the notion of parafactoriality  developed in \cite{EGAIV} 21.13.

\begin{definition}
Let $X$ be a locally noetherian scheme. Let $Y\subset X$ be a closed subset. Let $U=X-Y$. We say the pair $(X,Y)$ is \emph{parafactorial} if the restriction
\[
\PIC(X)\rightarrow \PIC(U)
\]
is an equivalence of categories.

A noetherian local ring $(R,\mathfrak{m})$ is \emph{parafactorial} if $(\Spec R,\{\mathfrak{m}\})$ is parafactorial.
\end{definition}
This implies that $\Pic(X)\cn \Pic(U)$, but the converse is not true in general.

\begin{proposition}\label{prop:para}
1. $(X,Y)$ is parafactorial if{f} for any $y\in Y$, $\mathcal{O}_{X,y}$ is parafactorial.

2. Assume $Y$ has codimension at least $2$ and $X$ is locally factorial, then $(X,Y)$ is parafactorial. 

3. Let $(R,\mathfrak{m})$ be a noetherian local ring, then $(R,\mathfrak{m})$ is parafactorial if{f}

a. $\Pic(\Spec R-\{\mathfrak{m}\})=0$.

and 

b. $\mathrm{depth}\, R\geq 2$.

4. Assume that $(X,Y)$ is parafactorial, then for any line bundle $L$ on $X-Y$, the natural morphism is an isomorphism
\[
i^*i_*L\cn L,
\]
where $i:X-Y\hookrightarrow X$ is the inclusion morphism.

5. Assume that the codimension of $Y$ in $X$ is at least $4$ and that $Y$ is locally complete intersection, then $(X,Y)$ is parafactorial.

6. Let $f:X\rightarrow S$ be a flat morphism between locally noetherian schemes. Let $x\in X$, $s=f(x)$. Suppose $\mathcal{O}_{X_s,x}$ is parafactorial and of depth $\geq 3$, then $\mathcal{O}_{X,x}$ is parafactorial.
\end{proposition}
\begin{proof}
For 1, see \cite{EGAIV} Proposition~21.13.10.

For 2, see \cite{EGAIV} Exemple~21.13.9(ii).

For 3, see \cite{EGAIV} Proposition~21.13.8.

For 4, see \cite{EGAIV} Proposition~21.13.5 and Lemme~21.13.2.

For 5, see \cite{SGAII} Expos\'{e}~XI Th\'{e}or\`{e}me~3.13. 

For 6, see \cite{SGAII} Expos\'{e}~XII Corollaire~4.8.
\end{proof}
Now we apply this proposition to the situation of Subsection~{subsec:univext}. Using the notations there.
\begin{corollary}\label{cor:para}
$(\mathbb{P},Z)$ is parafactorial if $n\geq 2$ or if $S$ is factorial.
\end{corollary}

\subsection{Review of determinant line bundles}\label{subsec:detline}
Fix $(f:X\rightarrow S)\in \mathcal{F}^n$ in this subsection.

Let $E$ be locally free sheaf of finite rank on $X$. Consider the complex $Rf_* E$. According to \cite{SGAVI} Expos\'{e}~3, Proposition~4.8, this complex is perfect. So according to \cite{KM77}, its determinant can be defined and is a line bundle on $S$. We write
\[
\lambda_{f}(E)=\lambda_{X/S}(E)=\lambda(E):=\det Rf_* E.
\]

When $n=0$, the determinant line bundle can be used to define a so-called \emph{norm} of a line bundle. 
\begin{definition}
Let $(f:X\rightarrow S)\in \mathcal{F}^0$. Let $L$ be a line bundle on $X$. Define the \emph{norm} of $L$ as the following line bundle on $S$:
\[
N_{f}(L)=N_{X/S}(L)=N(L):=\lambda(L)\otimes \lambda(\mathcal{O}_X)^{-1}.
\]
For the convention of the tensor product, see Appendix~
\ref{sec:appPic}.
\end{definition}
\begin{lemma}\label{lma:codim0lambda}
Let $(f:X\rightarrow S)\in \mathcal{F}^0$. Let $L_0$, $L_1$ be line bundles on $X$. There is a functorial isomorphism
\[
\lambda(L_0\otimes L_1)\otimes \lambda(\mathcal{O}_X)\cn \lambda(L_0)\otimes \lambda(L_1).
\]
\end{lemma} 
\begin{proof}
Observe that $f$ is actually a finite morphism according to \cite{EGAIV}, Corollaire~18.12.4. 

According to the functoriality of $\lambda$, we may assume that $S=\Spec A$ is affine and it suffices to construct an isomorphism
\[
\det H^0(X,L_0\otimes L_1) \otimes_A \det H^0(X,\mathcal{O}_X)\cn \det H^0(X,L_0)\otimes_A \det H^0(X,L_1).
\] 
Note that $X$ is an affine scheme, say $X=\Spec B$ for some finite $A$-algebra $B$. So we are reduced to Lemma~\ref{lma:detaffine}.
\end{proof}
\begin{lemma}\label{lma:detaffine}
Let $A$ be a commutative noetherian ring. Let $B$ be an $A$-algebra.
Let $M_0$, $M_1$ be finitely generated projective $B$-modules of rank $1$, then there is a functorial isomorphism
\[
\det_A (M_0\otimes_B M_1)\otimes_A \det_A B \cn \det_A M_0 \otimes_A \det_A M_1.
\]  
Moreover, if $M_2$ is another finitely generated projective $B$-module of rank $1$, the following diagram commutes
\[
\begin{tikzcd}
\det_A (M_0\otimes_B M_1 \otimes_B M_2)\otimes_A \det_A B \otimes_A \det_A B \arrow[d,"\sim"] \arrow[r,"\sim"]& \det_A M_0\otimes_A \det_A (M_1\otimes_B M_2)\otimes_A \det_A B \arrow[d,"\sim"]\\
\det_A(M_0\otimes_B M_1) \otimes_A \det_A M_2 \otimes \det_A B \arrow[r,"\sim"] &\det_A M_0 \otimes_A \det_A M_1 \otimes_A \det_A  M_2
\end{tikzcd}
\]
\end{lemma}
\begin{proof}
We may assume that $M_0$, $M_1$ are both free, the isomorphism is clear.

The commutative of the diagram is clear by definition.
\end{proof}

In particular, 
\begin{proposition}\label{prop:addN}
Let $(f:X\rightarrow S)\in \mathcal{F}^0$. Let $L_0$, $L_1$ be line bundles on $X$, then there exists functorial isomorphisms
\begin{equation}\label{eq:Ndist}
N_{X/S}(L_0\otimes L_1)\cn N_{X/S}(L_0)\otimes N_{X/S}(L_0).
\end{equation}
Moreover, the following diagram commutes
\[
\begin{tikzcd}
N(L_0\otimes L_1 \otimes L_2) \arrow[r,"\sim"] \arrow[d,"\sim"]& 
N(L_0\otimes L_1) \otimes N(L_2) \arrow[d,"\sim"]\\
N(L_0)\otimes N(L_1\otimes L_2) \arrow[r,"\sim"] & N(L_0)\otimes N(L_1)\otimes N(L_2)
\end{tikzcd}
\]
\end{proposition}
\begin{proof}
The existence of (\ref{eq:Ndist}) follows directly from Lemma~\ref{lma:codim0lambda}. The associativity follows from Lemma~\ref{lma:detaffine}.
\end{proof}

We shall use additive notations for line bundles for simplicity. 
\begin{lemma}\label{lma:natlambda}
Let $(f:X\rightarrow S)\in \mathcal{F}^0$. Let $g:S'\rightarrow S$ be a morphism from an irreducible noetherian scheme $S'$ to $S$. Let
\[
\begin{tikzcd}
X'  \arrow[d,"f'"] \arrow[r,"g'"]  & X  \arrow[d,"f"]\\
S' \arrow[r,"g"]  &S 
\end{tikzcd}
\]
be the corresponding Cartesian square. Then there is a functorial isomorphism
\[
\lambda_{X'/S'}(g'^*L)\cn g^* \lambda_{X/S}(L).
\]
\end{lemma}
\begin{remark}
It is crucial that we do not require $g$ to be flat.
\end{remark}
\begin{proof}
Locally, we may assume that $S=\Spec A$, $S'=\Spec A'$, $X=\Spec B$, $X'=\Spec B'$. $L$ is then identified with a finitely generated projective $B$-module of rank $1$. The problem reduces to establish
\[
\det_{A'} (M\otimes_B B')\cn \det_A M \otimes_A A'.
\]
But LHS equals
\[
\det_{A'} (M\otimes_B B')\cn \det_{A'}(M\otimes_B (B\otimes_A A'))\cn \det_{A'}(M\otimes_A A')\cn \det_{A} M \otimes_A A'.
\]
\end{proof}
\begin{corollary}\label{cor:cartnorm}
Let $(f:X\rightarrow S)\in \mathcal{F}^0$. Let $g:S'\rightarrow S$ be a morphism from an irreducible noetherian scheme $S'$ to $S$. Let
\[
\begin{tikzcd}
X'  \arrow[d,"f'"] \arrow[r,"g'"]  & X  \arrow[d,"f"]\\
S' \arrow[r,"g"]  &S 
\end{tikzcd}
\]
be the corresponding Cartesian square. Then there is a functorial isomorphism
\[
N_{X'/S'}(g'^*L)\cn g^*N_{X/S}(L).
\]
\end{corollary}

\begin{proposition}\label{prop:normpull}
Let $(f:X\rightarrow S)\in \mathcal{F}^0$. Then there is a canonical isomorphism
\[
N_{X/S}(f^*L)\cn \deg f L
\]
for line bundles $L$ on $S$.
\end{proposition}
\begin{proof}
It suffices to consider the case where $S$ is affine, say $S=\Spec A$. So $X$ is identified with $\Spec B$, where $B$ is a finite $A$-algebra. $L$ is identified with a finite generated projective $A$-module $M$ of rank $1$. Then $f^*L$ is identified with the $B$-module $M\otimes_A B$.
The theorem states that there is a canonical isomorphism
\[
\det{}_B(M\otimes_A B) \cn \deg(B/A)(\det{}_A M) \otimes_A B,
\]  
which is well-known.
\end{proof}
We now generalize Lemma~\ref{lma:codim0lambda} to general $n$.

\begin{lemma}\label{lma:lambdaind}
Let $(f:X\rightarrow S)\in \mathcal{F}^n$. Let $L$ be a line bundle on $X$. Let $D$ be an effective relative divisor on $X$. There is a functorial isomorphism
\[
 \lambda_{X/S}(\mathcal{O}_X(D)+L) \cn \lambda_{X/S}(L)+ \lambda_{D/S}(L|_D\otimes N_{D/X}),
\]
where $N_{D/X}$ is the normal bundle of $D$ in $X$.
\end{lemma}
\begin{proof}
Recall that we have the following exact sequence of sheaves on $X$,
\[
0\rightarrow  L \rightarrow L\otimes\mathcal{O}_X(D) \rightarrow \mathcal{O}_D(D)\otimes L|_D \rightarrow 0.
\]
This induces a long exact sequence of higher direct images. The desired isomorphism follows. 
\end{proof}
\begin{theorem}\label{thm:lambdarelation}
Let $(f:X\rightarrow S)\in \mathcal{F}^n$.

Let $(c_i)_{i=0,\dots,N}$ be integers. Then the following line bundle is canonically trivial
\begin{equation}\label{eq:sumlam}
\sum_{i=0}^N c_i \lambda(iL)
\end{equation}
for all line bundles $L$ on $X$
if the following relations holds
\begin{equation}\label{eq:ciident}
\sum_{i=0}^N c_i \binom{i}{j}=0,\text{ for } 0\leq j\leq n+1.
\end{equation}
\end{theorem}
This theorem seems to be a folklore result. It is known in a lot of literatures that $\lambda(iL)$ depends polynomially on $i$, see~\cite{KM77} for example, but it seems to the author that the exact relation \eqref{eq:ciident} has never been written down explicitly.
\begin{remark}
The inverse of this theorem is also true in the sense that (\ref{eq:sumlam}) is canonically trivial for all $f$ and all $L$ if{f} (\ref{eq:ciident}) holds.

According to the Mumford isomorphism, for a family of elliptic curves, $12\lambda(0)$ is always trivial.
In particular, the inverse of the theorem fails if we consider \emph{a single} family $f$ instead of all $f$.
\end{remark}
\begin{proof}
We make inductions on $n$. When $n=0$, this is a reformulation of Lemma~\ref{lma:codim0lambda}.

For $n>0$, first we deal with the case where $L$ is sufficiently ample.

By universal extension (Subsection~\ref{subsec:univext}) and the functoriality of $\lambda$, we may assume that $L=\mathcal{O}_X(D)$ for some given effective relative Cartier divisor $D$ on $X$. 
According to Lemma~\ref{lma:lambdaind}, there is a functorial isomorphism
\[
\sum_{i=0}^N c_i \lambda(iL) \cn \sum_{i=0}^N c_i \lambda(0) + \sum_{j=1}^N \sum_{i\geq j} c_i 
\lambda_D(j\tilde{L}|_{D}).
\]

This is canonically trivial for all $L$ if{f}
\begin{enumerate}
\item $\sum_i c_i=0$.
\item 
\[
 \sum_{j=1}^N \sum_{i\geq j} c_i 
\lambda_D(jL|_{D})
\]
is canonically trivial.
\end{enumerate}
By inductive hypothesis, the latter is true if{f}
\[
\sum_j \sum_{i\geq j}c_i \binom{j}{k}=0
\]
for any $0\leq k\leq n$.

Note that
\[
\sum_j \sum_{i\geq j}c_i \binom{j}{k}=\sum_i \sum_{j\leq i}\binom{j}{k} c_i=\sum_i \left( \binom{i}{k+1}+\binom{i}{k}\right)c_i
\]
The theorem follows in this case.

For a general $L$, assume that (\ref{eq:ciident}) holds. We have to prove that (\ref{eq:sumlam}) is canonically trivial.

Let $L'=\mathcal{O}_X(D)$ be a sufficiently ample line bundle such that $L+L'$ is also sufficiently ample. 
Then
\[
\sum_i c_i \lambda(iL+i\mathcal{O}_X(D)) \cn \sum_i c_i \lambda(iL)
\]
due to Lemma~\ref{lma:lambdaind} and (\ref{eq:ciident}).

By what we have already established, (\ref{eq:sumlam}) is trivial. We still have to show that this trivialization is independent of the choice of $L'$. Indeed, let $L''=\mathcal{O}(E)$ be another sufficiently ample line bundle on $X$, we have canonical isomorphisms
\[
\begin{tikzcd}
\sum_i c_i \lambda(iL+i\mathcal{O}_X(D)+i\mathcal{O}_X(E))\arrow[r,"\sim"]\arrow[d,"\sim"] &\sum_i c_i \lambda(iL+i\mathcal{O}_X(D))\arrow[d,"\sim"]\\
\sum_i c_i \lambda(iL+i\mathcal{O}_X(E)) \arrow[r,"\sim"] &\lambda(iL)
\end{tikzcd}
\]
This diagram commutes. (This can be seen by recalling the definitions of the isomorphisms in Lemma~\ref{lma:lambdaind}.) 

It follows that $L'+L''$ and $L'$ gives the same trivialization of (\ref{eq:sumlam}).
\end{proof}
\section{General theory of Deligne pairings}\label{sec:gen}
For the related concepts in Picard categories, we refer to Appendix~\ref{sec:appPic}.
\begin{definition}\label{def:DP}
A \emph{Deligne pairing} consists of the following data:
\begin{itemize}
\item For each $(f:X\rightarrow S)\in \mathcal{F}^n$, a functor 
\[
\langle\cdot\rangle_{X/S}=\langle\cdot\rangle_{f}\in L^{n+1}\left(\PIC(X),\PIC(S)\right).
\]
\item For each Cartesian square,
\begin{equation}\label{eq:Cart}
\begin{tikzcd}
X'  \arrow[d,"f'"] \arrow[r,"g'"]  & X  \arrow[d,"f"]\\
S' \arrow[r,"g"]  & S 
\end{tikzcd},
\end{equation}
where $f,f'\in \mathcal{F}^n$, we require a natural isomorphism between functors $\PIC(X)^{n+1}\rightarrow \PIC(S')$:
\[
\alpha_{f,g}:g^*\langle L_0,\dots,L_n \rangle_{X/S}\cn \langle g'^*L_0,\dots g'^*L_n \rangle_{X'/S'}.
\]
\item For each $(f:X\rightarrow S)\in \mathcal{F}^n$ for some $n>0$, a non-zero effective relative Cartier divisor $D$, a natural isomorphism between functors $\PIC(X)^{n} \rightarrow \PIC(S)$:
\[
\beta_{f,D}:\langle \mathcal{O}_X(D),L_1,\dots,L_n \rangle_{X/S}\cn \langle L_1|_{D},\dots, L_n|_{D} \rangle_{D/S}.
\]
\item
For each $(f:X\rightarrow S)\in \mathcal{F}^n$, a natural isomorphism between functors $\PIC(S)\times\PIC(X)^{n} \rightarrow \PIC(S)$:
\[
\gamma_{f}:\langle f^*L,L_1,\dots,L_n\rangle_{X/S}\cn L^{(L_1,\dots,L_n)},
\]
where $(L_1,\dots,L_n)$ is the intersection number of $L_1,\dots,L_n$ on any fibre. (See~\cite{Kol13} Appendix~2 for the precise definition.)
\item 
For $(f:X\rightarrow S)\in \mathcal{F}^0$, a natural isomorphism between functors $\PIC(X)\rightarrow \PIC(S)$
\[
\delta_f:\langle L \rangle_{X/S}\stackrel{\sim}{\rightarrow} N_{X/S}(L). 
\]
\end{itemize}

We require that $\beta,\gamma,\delta$ are natural with respect to base extension. 

To be precise: for a given Cartesian square (\ref{eq:Cart}), let $D$ be a non-zero effective relative Cartier divisor, $D'$ be its base extension to a divisor on $X'$ (which has the same properties by \cite{stacks-project} Tag~056P), such that
\[
\alpha_{f|_{D},g}^{-1}\circ \beta_{f',D'}\circ g'^{*}=g^*\circ \beta_{f,D}.
\]
And similar identities hold for $\gamma$ and $\delta$.
\end{definition}
\begin{remark}\label{rmk:nDP}
When in this definition, $\mathcal{F}^n$ is replaced by $\mathcal{G}^n$, namely, the class of $(f:X\rightarrow S)\in \mathcal{F}^n$, such that $S$ is locally noetherian, we say we have a noetherian Deligne pairing. We shall prove later on that each noetherian Deligne pairing extends automatically to a Deligne pairing. (Proposition~\ref{prop:extnondp})
\end{remark}

\begin{remark}
When a good intersection theory on $X$ and $S$ is defined, the Deligne pairing indeed represents the push-forward of intersection products. More precisely,

The following diagram is commutative:
\begin{equation}
\begin{tikzcd}
\PIC(X)^{n+1}/\mathfrak{S}_{n+1} \arrow[r, "\langle\cdot \rangle_{X/S}"] \arrow[d,"c_1"]  & \PIC(S)\arrow[d,"c_1"] \\
\CH^1(X)^{n+1}/\mathfrak{S}_{n+1} \arrow[r,"\int"]  & \CH^1(S)
\end{tikzcd},
\end{equation}
where $\int$ denotes the push-forward of the product in Chow ring.

This is the case for example when $X$, $Y$ are algebraic schemes over a fixed field and $X$ is geometrically regular or simply when $X$ is regular, in which case, we actually have to use the rational Chow groups instead.

This fact follows from the $\beta$ isomorphism, additivity of Deligne pairings and Lelong-Poincar\'{e} formula.
\end{remark}

\subsection{Uniqueness}
\begin{theorem}\label{thm:uniqueness}
The Deligne pairing is unique in the following sense: Given Degline pairings say $(\langle\cdot \rangle^i,\alpha^i,\beta^i,\gamma^i,\delta^i)$ ($i=1,2$), there is a unique additive natural isomorphism $F:\langle\cdot \rangle^1 \rightarrow \langle\cdot \rangle^2$ that transforms $\alpha$, $\beta$, $\gamma$, $\delta$ accordingly.
\end{theorem}
\begin{proof}
Let $(f:X\rightarrow S)\in \mathcal{F}^n$. 
Let $L_0,\dots,L_n\in \PIC(X)$. We want to construct in a functorial way a symmetric additive isomorphism
\[
F^X_{L_0,\dots,L_n}:\langle L_0,\dots,L_n \rangle^1_{X/S}\rightarrow \langle L_0,\dots,L_n \rangle^2_{X/S}.
\]

We make an induction on $n$. 

When $n=0$, this is trivially done using $\delta$.

When $n>0$, first we consider the case where $L_0$ is sufficiently ample.

Consider the universal construction(Subsection~\ref{subsec:univext}):
\begin{equation}
\begin{tikzcd}
V:=\mathbb{P}\times_S X \arrow[d,"\pi_1"]\arrow[r,"\pi_2"] & X \arrow[d,"f"] \\
\mathbb{P} \arrow[r,"\pi"] &S
\end{tikzcd}
\end{equation}
Here $\mathbb{P}=\mathbb{P}_S(f_*L_0)^\vee$. We know that there is a canonical line bundle $\mathcal{L}$ on $V$ together with a canonical section $\sigma$, which is flat relative to $\pi_1$ outside a codimension $2$ subset $Z\subset \mathbb{P}$. (Proposition~\ref{prop:outcod2}). We shall write $Z'(\sigma)=Z(\sigma)-\pi^{-1}_1(Z)$.
 
An isomorphism between line bundles $\langle L_0,\dots,L_n\rangle_{X/S}^i$ is equivalent to an isomorphism between
 $\pi|_{\mathbb{P}-Z}^*\langle L_0,\dots,L_n\rangle_{X/S}^i$
 by \cite{EGAIV} Lemme~21.13.2.

Applying $\alpha$ isomorphisms, this is further equivalent to an isomorphism between
\[
\langle \pi'^{-1}L_0,\dots,\pi'^{-1}L_n\rangle_{V-\pi_1^{-1}(Z)/\mathbb{P}-Z}^i,
\]
where $\pi'$ is the restriction of $\pi_2$ to $V-\pi_1^{-1}(Z)$.

Using the additivity morphism and the $\gamma$ isomorphism, this is further equivalent to an isomorphism between
\[
\langle \mathcal{L},\pi'^*L_1,\dots, \pi'^*L_n \rangle_{V-\pi_1^{-1}(Z)/\mathbb{P}-Z}^i.
\]

Using $\beta$ isomorphism, this is finially equivalent to an isomorphism between
\[
\langle \pi'^*L_1,\dots, \pi'^*L_n \rangle_{Z'(\sigma)}^i,
\]
where $\sigma$ is the universal section of $\mathcal{L}$, which is uniquely determined by induction.

Now for general $L_0$, take a sufficiently ample line bundle $H$ so that $L_0+H$ is also sufficiently ample.
We then define
\[
F^X_{L_0,L_1,\dots,L_n}
\]
as the difference of $F^X_{L_0+H,L_1,\dots,L_n}$ and $F^X_{H,L_1,\dots,L_n}$.

This construction is independent of the choice of $H$. To see this, it suffices to show $F^X_{L_0,\dots,L_n}$ is additive with respect to sufficiently ample $L_0$.

So now we assume that $L_0^1$ and $L_0^2$ are sufficiently ample line bundles on $X$. 
Consider the universal constructions:
\begin{equation}
\begin{tikzcd}
V^j:=\mathbb{P}^j\times_S X \arrow[d,"\pi_1"]\arrow[r,"\pi_2"] & X \arrow[d,"f"] \\
\mathbb{P}^j \arrow[r,"\pi"] &S
\end{tikzcd},
\end{equation}
where $\mathbb{P}^j=\mathbb{P}(f_*L_0^j)^\vee$. We write $\mathcal{L}^j$ and $\sigma^j$ for the universal bundles and universal sections.
Define $L_0=L_0^1+L_0^2$.

We also have another universal construction
\begin{equation}
\begin{tikzcd}
V:=\mathbb{P}\times_S X \arrow[d,"\pi_1"]\arrow[r,"\pi_2"] & X \arrow[d,"f"] \\
\mathbb{P} \arrow[r,"\pi"] &S
\end{tikzcd},
\end{equation}
where 
\[
\mathbb{P}=\mathbb{P}\left((f_*L_0^1)^\vee \otimes (f_*L_0^2)^\vee\right).
\]
We write $\mathcal{L}$ and $\sigma$ for the universal bundles and universal sections.

Recall that there is a natural embedding
\[
i:\mathbb{P}^1\times \mathbb{P}^2\hookrightarrow \mathbb{P}.
\]
We write $i$ also for the corresponding map of $V$'s.

We find
\[
i^*\mathcal{L}=\mathcal{L}^1\boxtimes \mathcal{L}^2.
\]
By definition, we find that
\[
i^*\sigma=\sigma^1\boxtimes \sigma^2.
\]
Hence
\begin{equation}\label{eq:Zaddi}
Z(\sigma)=Z(\sigma^1)+Z(\sigma^2)
\end{equation}
as divisors on $X\times_S (\mathbb{P}^1\times_S \mathbb{P}^2)$, where obvious pull-back maps are omitted. We conclude by \cite{Elk89} I.2.9 and I.2.4.

The behaviors of $\alpha$, $\beta$, $\gamma$ are easy to see. We prove this for $\alpha$ for example. Notations are as in (\ref{eq:Cart}).
In this case, we need a commutative diagram as
\[
\begin{tikzcd}
g^*\langle L_0,\dots,L_n\rangle_{X/S}^1 \arrow[r,"\sim"]\arrow[d,"\sim"] &\langle g'^*L_0,\dots,g'^*L_n\rangle_{X'/S'}^1 \arrow[d,"\sim"] \\
g^*\langle L_0,\dots,L_n\rangle_{X/S}^2 \arrow[r,"\sim"] & \langle g'^*L_0,\dots,g'^*L_n\rangle_{X'/S'}^2
\end{tikzcd}.
\]
Using the same method as above, we may assume $L_0=\mathcal{O}(D)$ for some non-zero effective Cartier divisor on $X$. According to the base change invariance of $\beta$, we reduce immediately to $(n-1)$ dimensional case. So finally, we are left with the case $n=0$, which amounts to the base change invariance of $\delta$.
\end{proof}

We conclude a very important principle from this proof:
\begin{principle}\label{prin:verif}
In order to verify a universal functorial relation concerning Deligne pairings of line bundles, we may assume some of the line bundles under consideration are sufficiently ample and have the form $\mathcal{O}(D)$ for some non-zero effective relative Cartier divisors $D$ that are prescribed \emph{a priori}.
\end{principle}
In practice, we find that it is much easier to apply this principle instead of applying the uniqueness theorem itself.
\subsection{Existence}\label{subsec:exis}
\begin{theorem}
Deligne paring exists.
\end{theorem}
This has been proved in some literatures, like \cite{Elk89}, \cite{Duc05}, \cite{Gar00}. 

We reproduce a proof following \cite{Elk89} for the convenience of readers.

Before starting the proof, we shall remark on the difference between the construction here and that in
\cite{Elk89}. 

1. In \cite{Elk89}, Elkik assume that $f$ is Cohen-Macaulay, which is not necessary. She used this for the proof of Lemme I.1.3. Following \cite{Gar00}, we do not need that lemma, we use Proposition~\ref{prop:outcod2} instead.

2. We establish Deligne pairings also for non-noetherian schemes, this is a simple application of the limit procedure developed in \cite{EGAIV} Section~8 and Section~11.
Although quite simple, it seems that these constructions have not been written down in the existing literatures.

Our proof proceeds in two steps,
the first step is to establish a noetherian Deligne pairing (Remark~\ref{rmk:nDP}).

\begin{lemma}
Noetherian Deligne pairing exists.
\end{lemma}
\begin{proof}
Let $(f:X\rightarrow S) \in \mathcal{G}^n$. 

Let $L_0,\dots,L_n$ be line bundles on $X$. We proceed in several steps.

Step 1. Construction of $\langle L_0,\dots,L_n \rangle_{X/S}$ under the assumption that $L_i$'s are sufficiently ample for $i\neq 0$. In that case, we set
\[
\mathbb{P}=\prod_{i=1}^n \mathbb{P}_i,
\]
where
\[
\mathbb{P}_i=\mathbb{P}(f_*L_i)^\vee.
\]
Let $\mathcal{L}_i$, $\sigma_i$ be the corresponding universal bundles and universal sections(See Subsection~\ref{subsec:univext}). 
It follows from \cite{Elk89} Theorem~I.2.4 that there is a closed subset $Z\subset \mathbb{P}$ of codimension  $2$, such that 
$\bar{Z}=\cap_i Z(\sigma_i)$ is finite over $U:=\mathbb{P}-Z$. 
Here and in the sequel, obvious pull-back maps are always omitted.

\begin{equation}
\begin{tikzcd}
L_0 \arrow[d,dashed] &\mathcal{L} \arrow[d,dashed]
& \mathcal{L} \arrow[d,dashed]&\mathcal{L}_i \arrow[d,dashed]& L_0,\dots,L_n \arrow[d,dashed]\\
\bar{Z} \arrow[r] \arrow[d] & X\times_S U \arrow[r] \arrow[d] & X\times_S \mathbb{P} \arrow[r] \arrow[d] & X\times_S \mathbb{P}_i \arrow[r] \arrow[d] & X\arrow[d,"f"] \\
U \arrow[r,"\text{id}"]& U\arrow[r,hookrightarrow] &\mathbb{P} \arrow[r] & \mathbb{P}_i \arrow[r,"p_i"] &S
\end{tikzcd}
\end{equation}

We now define 
\[
\langle L_0,\dots,L_n \rangle_{X/S}
\]
as the descendent of $N_{\bar{Z}/U}(L_0)$ to $S$. To be precise, by successively descending to
\[
\mathbb{P}_1\times \dots \times \mathbb{P}_j 
\]
for $j=n-1,\dots,0$ using Proposition~\ref{prop:para} 6, we reduce immediately to the case $n=1$, which is established in \cite{Elk89} Lemme~I.2.8.

This definition makes $L_0$ special among $L_i$'s. We shall check in several steps that $L_0$ is symmetric with all other $L_i$'s.

1) Let $(s_i)_{i\in J}$ be a sequence of sections of $L_i$, where $J$ is a subset of $\{1,\dots,d\}$. Assume that $Y:=\cap Z(s_i)$ is flat over $S$. We construct a canonical isomorphism
\[
\beta:\langle L_0,L_1,\dots,L_n\rangle_{X/S} \rightarrow \langle L_0, \mathbb{L}\rangle_{Y/S},
\]
where $\mathbb{L}$ is the sequence obtained from $L_1,\dots,L_n$ by deleting those indexed by $J$.

We do so by assuming $J=\{1\}$, so we drop the subindex $i$. The general case is similar.
In this case, it suffices to observe that $s:\mathcal{O}_X \rightarrow L_1$ induces an $S$-morphism
$S\rightarrow \mathbb{P}_1$.
This further induces the following commutative diagram:
\[
\begin{tikzcd}
\bigcap_{i=2}^{n} Z(\sigma_2) \arrow[r]\arrow[d] &\bar{Z}\arrow[d]\\
\mathbb{P}_2\times_S \cdots\times_S \mathbb{P}_n \arrow[r,"i"] &\mathbb{P}
\end{tikzcd}
\]
It is immediate that this diagram is Cartesian, and in particular, Cartisian when restricted to $U$ and $i^{-1}U$.

The desired isomorphism is induced by Proposition~\ref{cor:cartnorm} and Theorem~\ref{thm:desc}.

2) Let $(s_i)_{i=0,\dots,d}$ be a sequence of sections of $L_i$, assume that $\bar{Z}:=\cap_{i=1}^d Z(s_i)$ is flat over $S$. Then by 1), we have an isomorphism
\[
\beta:\langle L_0,\dots,L_n\rangle_{X/S}\rightarrow N_{\bar{Z}/S}(L_0).
\]
On an open subset $f^{-1}(W)\subset \bar{Z}$ where $s_0$ is nowhere zero, we have a section $N_{W/S}(s_0)$ of $N_{Z/S}(L_0)$, hence the inverse image under beta becomes a nowhere zero section of $\langle L_0,\dots,L_n\rangle_{X/S}$, which we denote as 
\[
\langle s_0,\dots,s_n\rangle_{X/S}.
\]

3) The symmetric morphism: Assume in this paragraph that $L_0$ is also sufficiently ample.
Given a permutation of $\{0,\dots,d\}$, say $\sigma$, we have an isomorphism
\[
\langle L_0,\dots,L_n\rangle_{X/S}\rightarrow
\langle L_{\sigma 0},\dots,L_{\sigma n}\rangle_{X/S}.
\]
The construction is obvious if $\sigma 0=0$. So we are reduced to the following case: $n=1$, $\sigma$ exchanges $0$ and $1$.

We shall use $\approx$ to denote an isomorphism without counting torsions given by $\mathcal{O}(n)$ sheaves. We shall omit all obvious pull-back maps.

Consider the following diagram on $U\times_S X$:
\begin{equation}
\begin{tikzcd}
           & 0 & 0 & &\\
0\arrow[r] & \mathcal{L}_0|_{Z(\sigma_0)}\arrow[u]\arrow[r,"\sigma_1"] & \mathcal{L}_0|_{Z(\sigma_0)}\otimes \mathcal{L}_1 \arrow[u] \arrow[r] &0 &\\
0 \arrow[r] &\mathcal{L}_0 \arrow[u]\arrow[r,"\sigma_1"] &\mathcal{L}_0\otimes \mathcal{L}_1 \arrow[u]\arrow[r]&\mathcal{L}_0\otimes\mathcal{L}_1|_{Z(\sigma_1)}\arrow[u]\arrow[r] &0 \\
0 \arrow[r] & \mathcal{O}_X \arrow[u,"\sigma_0"]\arrow[r,"\sigma_1"] &\mathcal{L}_1 \arrow[u,"\sigma_0"]\arrow[r]&\mathcal{L}_1|_{Z(\sigma_1)}\arrow[u,"\sigma_0"]\arrow[r] &0\\
 &0 \arrow[u]&0\arrow[u]&0\arrow[u] &
\end{tikzcd}
\end{equation}
Let $\mathcal{S}$ be the determinant line bundle of the lower left square, namely
\[
\mathcal{D}=\lambda(\mathcal{L}_0\otimes\mathcal{L}_1)\otimes \lambda(\mathcal{L}_1)^{-1} \otimes \lambda(\mathcal{L}_0)^{-1} \otimes \lambda(\mathcal{O}_X).
\]
There are natural isomorphisms
\[
\mathcal{D}\stackrel{\sim}{\rightarrow} \lambda(\mathcal{L}_0\otimes \mathcal{L}_1|_{Z(\sigma_1)})\otimes 
\lambda(\mathcal{L}_1|_{Z(\sigma_1)})^{-1}
\stackrel{\sim}{\rightarrow}
N_{Z(\sigma_1)/\mathbb{P}}(\mathcal{L}_0)\approx \langle L_0,L_1\rangle_{X/S}
.
\]
and
\[
\mathcal{D}\stackrel{\sim}{\rightarrow} \lambda(\mathcal{L}_0|_{Z(\sigma_0)}\otimes \mathcal{L}_1)\otimes 
\lambda(\mathcal{L}_0|_{Z(\sigma_0)})^{-1}
\stackrel{\sim}{\rightarrow}
N_{Z(\sigma_0)/\mathbb{P}}(\mathcal{L}_1)\approx \langle L_1,L_0\rangle_{X/S}
.
\]
Here $\lambda$ means $\lambda$ from different sets to a subset of $\mathbb{P}$. The middle isomorphism is that in Lemma~\ref{lma:codim0lambda}.

From these isomorphisms, we get the desired 
\begin{equation}\label{eq:isoperm}
\langle L_0 ,L_1\rangle_{X/S}\rightarrow \langle L_1,L_0\rangle_{X/S}.
\end{equation}
There is a little problem with sign here. Namely, the trivializations of $\mathcal{D}$ defined by $N\sigma_1$ and $N\sigma_0$ are different. To be precise, let $r_0$ be the number of $\mathcal{O}(1)$ torsions there are in $N_{Z(\sigma_1)/\mathbb{P}}(\mathcal{L}_0)$, and similarly define $r_1$. Then the two trivializations differ by a factor $(-1)^{r_1 r_0}$.(See, for example, \cite{Elk89} Page~207) We correct the isomorphism (\ref{eq:isoperm}) by this factor.

In particular, we find, using the notations of 2), that
\[
\langle s_0,s_1\rangle_{X/S}\rightarrow \langle s_1,s_0\rangle_{X/S}
\] 
under this isomorphism.

4) Additivity morphisms:
Additivity with respect to $L_i$, $i\neq 0$. This is essentially constructed in the proof of uniqueness, see in particular (\ref{eq:Zaddi}).

Additivity with respect to $L_0$ is induced by the additivity of $N$ proved in Proposition~\ref{prop:addN}.

It is immediate that the sections constructed in 2) is additive with respect to all $s_i$.

5) When $L_0=f^*L$ for some line bundle $L$ on $S$, we need an isomorphism 
\[
\gamma:\langle L_0,L_1,\dots,L_n\rangle_{X/S}\stackrel{\sim}{\rightarrow} L^{N}.
\]
This has been constructed in Proposition~\ref{prop:normpull}.

Note that if the generic fibre of $f$ is regular, we could conclude that $N$ is equal to the intersection number of $L_1,\dots,L_N$ at the generic fibre.

Now we have a number of axioms to verify. More precisely:

a) The additivity morphism is compatible with the symmetric morphism. 

b) These constructions are natural with respect to base extension.

c) The additivity morphisms satisfy the compatibility condition \ref{eq:comp}.

Among them, b) and c) are easy. We only prove a) here. The non-trivial part can be reduced as follows: assume further $n=1$, $\sigma(0)=1$, $L_0$, $L_1$, $L_1'$ are sufficiently ample line bundles on $X$. We need to prove the following diagram commutes:
\[
\begin{tikzcd}
\langle L_0,L_1\otimes L_1' \rangle_{X/S}\arrow[r,"a"] \arrow[d,"b"]& \langle L_0,L_1 \rangle_{X/S}\otimes \langle L_0,L'_1 \rangle_{X/S}\arrow[d,"c"]\\
\langle L_1\otimes L_1',L_0 \rangle_{X/S}\arrow[r,"d"]&\langle L_1,L_0 \rangle_{X/S}\otimes \langle L'_1,L_0 \rangle_{X/S}
\end{tikzcd}
\]
It suffices to show that the local sections of $\langle L_0,L_1\otimes L_1' \rangle_{X/S}$ have the same images in $\langle L_1,L_0 \rangle_{X/S}\otimes \langle L'_1,L_0 \rangle_{X/S}$ under both maps. This follows from what we have established in 3) and 4) above.

Step 2. Extension to general $L_i$'s. This is already done in the proof of uniqueness. Namely, writing
\[
L_i=L_i^0\otimes (L_i^1)^{-1},
\] 
where $i>0$, $L_i^j$ are sufficiently ample line bundles on $X$.
 
Set $P=\{0,1\}^n$.
We then define
\[
\langle L_0,L_1,\dots, L_n \rangle_{X/S}:=\prod_{I\in P} \langle  L_0,L_1^{I(1)},\dots, L_1^{I(n)}\rangle_{X/S}^{(-1)^{\sum_j P(j)}}.
\]
This construction is independent of the choice of $L_i^j$ up to unique isomorphism as shown in the proof of uniqueness. We fix one choice.

All requirements in the definition of Deligne pairings now follows easily from what we have established.
\end{proof}
\begin{proposition}\label{prop:extnondp}
Each noetherian Deligne pairing extends naturally to a Deligne pairing.
\end{proposition}

\begin{proof}
Let $(\langle \cdot \rangle,\alpha,\beta,\gamma,\delta)$ be a noetherian Deligne pairing.

Let $(f:X\rightarrow S)\in \mathcal{F}^n$. Let $L_0,\dots, L_n$ be line bundles on $X$. Recall that $\PIC$ is an Artin stack, so we may assume further that $S$ is affine, say $S=\Spec A$. It follows from \cite{EGAII} Proposition~8.9.1, Proposition~8.5.5, Th\'{e}or\`{e}me~8.5.2, Th\'{e}or\`{e}me~8.10.5, Corollaire~11.2.7 that there is a subring $A_0$ of $A$, which is of finite type over $\mathbb{Z}$ and a projective flat morphism $f_0:X_0\rightarrow \Spec A_0$ between noetherian schemes such that $f$ is the base extension of $f_0$ along the natural morphism $j:\Spec A\rightarrow \Spec A_0$ 
\[
\begin{tikzcd}
X \arrow[r,"j'"]\arrow[d,"f"] & X_0\arrow[d,"f_0"]\\
\Spec A \arrow[r,"j"] & \Spec A_0
\end{tikzcd}
\]
and there exists line bundles $L_i^0$ on $X_0$ such that there is a canonical isomorphism
\[
\varphi_i:L_i\cn j'^{*}L_i^0.
\]
Examining the proof of \cite{EGAII} Th\'{e}or\`{e}me~8.5.2(ii) and using the fact that quasi-coherent sheaves form an Artin stack (\cite{Lau99} Th\'{e}or\`{e}me~4.6.2.1), it is not hard to see that there is a functorial choice of $L_i^0$ when $A_0$ is fixed. 

Now we define
\[
\langle L_0,\dots,L_n\rangle_{X/\Spec A}:=j^*\langle L_0^0,\dots,L_n^0 \rangle_{X_0/\Spec A_0}.
\]

This construction depends on the choice of $A_0$ in the following manner: for two different choices of $A_0$, there is a canonical isomorphism between the corresponding Deligne pairings, as stated in Appendix~\ref{sec:appPic}, we may fix a choice so that the definition is completely fixed.

The base extension invariance follows from \cite{EGAII} Proposition~8.9.1 and the corresponding result in noetherian case. 

The $\beta$ isomorphism is established by \cite{EGAII} Th\'{e}or\`{e}me~8.5.2 and Corollaire~11.2.7.

The $\gamma$ isomorphism and the fact that the pairing is symmetric multilinear follow from \cite{EGAII} Th\'{e}or\`{e}me~8.5.2.

The $\delta$ isomorphisms follows from Lemma~\ref{lma:natlambda}.

Finally, the compatibility of these isomorphisms follows directly from the noetherian case.
\end{proof}
\subsection{Generalization to general Chern polynomials}
In this subsection, following \cite{Elk89}, we generalize the construction in the previous section to more general Chern polynomials.

Let $(f:X\rightarrow S)\in \mathcal{F}^n$.

Let $E_i$ ($1\leq i\leq N$) be vector bundles of rank $r_i+1$. Let $P$ be a homogeneous polynomial of $c_j(E_i)$ of degree $n+1$ with rational coefficients. ($c_j$ has degree $j$.)

We shall define a line bundle $\langle P(c_j(E_i) \rangle_{X/S}$ on $S$ functorially, such that
\[
c_1\langle P(c_j(E_i) \rangle_{X/S}= f_* P(c_j(E_i)).
\]

The idea is: Segre class has a natural intersection theoretic meaning, so one may try to pull everything back to a projectivilization of $E_i$ and reduce to the line bundle case. This indeed works.

Set $\mathbb{P}=\mathbb{P}(E_1)\times \cdots \times \mathbb{P}(E_m)$.

\begin{definition}
The Deligne pairing of $P(c_j(E_i))$ is a line bundle $\langle P(c_j(E_i)) \rangle_{X/S}$ on $S$ with metric, defined by linear extension of
\begin{equation}\label{eq:extdelg}
\langle s_{k_1}(E_1)\dots s_{k_m}(E_m)\rangle_{X/S}:=\langle \mathcal{O}_{\mathbb{P}E_1}(1)\{r_1+k_1\},\dots,\mathcal{O}_{\mathbb{P}E_m}(1)\{r_m+k_m\} \rangle_{\mathbb{P}/S},
\end{equation}
where $s_j$ denotes the $j$-th Segre class, $\sum_j k_j=n+1$.
\end{definition}
An easy calculation shows
\begin{proposition} Then Chern class of Deligne pairing is given by
\[
c_1(\langle P(c_i(E_j)) \rangle_{X/S})=f_* P(c_i(E_j)).
\]
\end{proposition}
In the following $Q$ will denote a Chern polynomial of $E_i$ of proper degree.
\begin{proposition}\label{prop:seqdelign}
Let $0\rightarrow E\rightarrow F\rightarrow G \rightarrow 0$ be an exact sequence of vector bundles on $X$, then there is a functorial isomorphism
\[
\langle c_j(F)Q\rangle_{X/S}\cn \sum_{a+b=j}\langle c_a(E)c_b(G) Q\rangle_{X/S}
\]
Similar result holds for Todd classes.
\end{proposition}
\begin{proposition}
Let $D$ be a non-zero effective relative Cartier divisor on $X$. Then there is a functorial isomorphism
\[
\langle c_1(\mathcal{O}_X(\mathcal{D})) Q\rangle_{X/S}\cn \langle Q \rangle_{D/S}
\]
\end{proposition}
\begin{proposition}
Let $L_0,\dots,L_n$ be line bundles on $X$. Then
\[
\langle L_0,\dots,L_n\rangle_{X/S}\cn \langle c_1(L_0)\dots c_1(L_n)\rangle_{X/S}.
\]
\end{proposition}
For the proofs, we refer to \cite{Elk90} and \cite{Elk89}.

We note that Principle~\ref{prin:verif} still holds for these general Deligne pairings.

\section{Proof of Deligne-Riemann-Roch theorem}\label{sec:DRR}
In this section, we prove a higher dimensional Deligne-Riemann-Roch theorem.

\begin{theorem}\label{thm:DRR}
Let $(f:X\rightarrow S)\in \mathcal{F}^n$. Assume that  $f$ is l.c.i..
Let $E$ be a vector bundle on $X$. Then there is an invariant  functorial isomorphism between functors $\Vect(X)\rightarrow \PIC_{\mathbb{Q}}(S)$
\begin{equation}\label{eq:DRR}
\lambda_{X/S}(E)\cn \langle \RR^{n+1}(E)\rangle_{X/S}.
\end{equation}

Moreover, this isomorphism is compatible with short exact sequences.
\end{theorem}

For the definition of $\RR$, we refer to Subsection~\ref{subsec:rrp}.
For the precise meaning of compatibility with short exact sequences, see~Subsection~\ref{subsec:compati}.

\subsection{Riemann-Roch polynomials}\label{subsec:rrp}
Let $(f:X\rightarrow S)\in \mathcal{F}^n$. Assume that $f$ is l.c.i..
Let $E$ be a vector bundle on $X$.

Define $\RR(E)$ as the Riemann-Roch polynomial of $E$, namely
\[
\RR(E)=\RR_{X/S}(E):=\Td(T_{X/S})\ch(E).
\]
And $\RR^{a}(E)$ denotes the homogeneous part of degree $a$ of $\RR(E)$.

We explain the meaning of $\langle  \RR_{X/S}^{n+1}(E) \rangle_{X/S}$.

 \cite{Knud02}

Moreover, it makes sense to talk about $\langle [\RR^{n+1}(E^\bullet)\rangle_{X/S}$ for any perfect complex $E^\bullet$ for the same reason.

\begin{proposition}\label{prop:RRrest}
$f$, $E$ as above. Let $D$ be an effective relative divisor on $X$. Then
\[
i^*\RR_{X/S}(E)=\RR_{D/S}(E)\sum_{j=0}^{\infty}\frac{(-1)^j}{j!}  B_j i^*c_1^j(\mathcal{O}_X(D)),
\]
where $i:D\rightarrow X$ is the inclusion morphism.
\end{proposition}
Here $B_j$ denotes the Bernoulli numbers. See Appendix~\ref{sec:appCom} for the sign convention.
\begin{proof}
Recall that there is an exact sequence
\[
0\rightarrow T_{D/S}\rightarrow i^*T_{X/S} \rightarrow \mathcal{O}_D(D)\rightarrow 0.
\]
Hence we find
\[
i^*\Td(T_{X/S})=\Td(T_{D/S})\Td(\mathcal{O}_D(D)).
\]
\end{proof}
\begin{remark}
	Observe that the restriction of $f$ to a non-zero effective relative Cartier divisor $D$ is still l.c.i.. This follows easily from the fact that a flat l.c.i. morphism is syntomic. (\cite{stacks-project} Tag~069K)
\end{remark}

\subsection{Reduction to line bundles}
We observe that the proof can be reduced to the case where $E=L$ is a line bundle. 

Let $r>1$ be the rank of $E$. We make an induction on $r$.  

Now consider the projectivilization of $E$, 
$\pi:\mathbb{P}E \rightarrow X$.

Note that on $\mathbb{P}E$, there is a short exact sequence of vector bundles
\[
0\rightarrow \mathcal{O}(-1)\rightarrow \pi^*E\rightarrow Q\rightarrow 0,
\]
where $Q$ is the universal quotient bundle.

We construct the isomorphism in (\ref{eq:DRR}) for $\pi^* E$ from those for $\mathcal{O}(-1)$ and $Q$.  

While we calculate
\[
\RR_{\mathbb{P}E/S}(\pi^*E)=\Td(T_{\mathbb{P}E/X})\pi^*\RR_{X/S}(E).
\] 
Hence
\[
\lambda_{X/S}(E)\cn \lambda_{\mathbb{P}E/S}(\pi^*E)
\cn \langle \RR^{n+r}(\pi^*E)\rangle_{\mathbb{P}E/S}
\cn  \langle \RR^{n+1}_{X/S}(E)\rangle_{X/S},
\]
where the first isomorphism follows from the usual projection formula, the second isomorphism is already constructed. While the third follows from 
the calculation
\[
\RR^{n+r}_{\mathbb{P}E/S}(\pi^*E)=\Td^{r-1}(T_{\mathbb{P}E/X})\pi^*\RR^{n+1}_{X/S}(E)
=c_1^{r-1}(\mathcal{O}_{\mathbb{P}E}(1))\pi^*\RR^{n+1}_{X/S}(E).
\]
So the third isomorphism follows from the definition of Deligne pairing of general  Chern polynomials, see (\ref{eq:extdelg}). 

We shall prove in Subsection~\ref{subsec:compati} that  (\ref{eq:DRR}) is compatible with exact sequences.

We assume throughout this section that $L$ is a line bundle on $X$.

\subsection{The case of projective spaces}\label{subsec:projsp}
At first, let us prove the theorem for projective space bundles.

Let $\pi:\mathbb{P}^n_S\rightarrow S$ be a projective space over $S$. We are going to prove that
\begin{equation}\label{eq:proj}
\lambda_{\mathbb{P}^n_S/S}(L)\cn \langle [ \Td(T_{\pi})\ch(L)]^{n+1}  \rangle_{\mathbb{P}^n_S/S}.
\end{equation}

We first deal with the case where $L=\mathcal{O}(d)$ for some $d\in \mathbb{Z}$.

In this case, $\lambda(L)$ is trivial.
While clearly RHS is a constant multiple of 
\[
\langle c_1^{n+1}(\mathcal{O}(1))\rangle_{\mathbb{P}^n_S/S},
\]
which is trivial by induction on $n$.

So now we may assume that $L=\pi^* L'$ for some line bundle $L'$ on $S$.
In this case, LHS of (\ref{eq:proj}) is canonically identified with $L'$ by projection formula. As for RHS, expand it according to degrees of both parts. Clearly, the only possibly non-trivial part is given by
\[
\langle\Td^n(T_{\pi})\pi^* L' \rangle_{\mathbb{P}^n_S/S}\cn L'
\]
by the $\gamma$ isomorphism.

\subsection{$n=0$ case}
When $n=0$, we claim that (\ref{eq:DRR1}) holds with some $U_{X/S}$. 

Actually, in this case, we have
\[
\lambda_{X/S}(L)-\lambda_{X/S}(0)\cn \langle L \rangle_{X/S} \cn \langle [\Td(T_{X/S})\ch(L)]^1 \rangle_{X/S}-\langle \Td^1(T_{X/S})\rangle_{X/S}.
\]

This implies that $U_{X/S}$ does not depend on $L$.

Now consider an embedding
\[
\begin{tikzcd}
X \arrow[r,"i",hookrightarrow] \arrow[dr,"f"]& \mathbb{P}V \arrow[d,"\pi"] \\ 
&S
\end{tikzcd},
\]
where $V$ is a vector bundle on $S$ of rank $N+1$.

Then we find
\[
\lambda_{X/S}(0)\cn \lambda_{\mathbb{P}V/S}(i_*\mathcal{O}_X) \cn  \langle \RR^{N+1}(i_*\mathcal{O}_X) \rangle_{\mathbb{P}V/S}\cn \langle \RR^{n+1}(0) \rangle_{X/S},
\]
where the last isomorphism follows from the $\beta$-isomorphism of Deligne pairings using the Koszul resolution of $i_*\mathcal{O}_X$. 

In the following, we shall assume that $n>0$ and the theorem is already proved for dimension up to $n-1$.
\subsection{Inversion of (\ref{eq:DRR})}
The first step is to prove: There is a canonical isomorphism that commutes with base extension
\begin{equation}\label{eq:DRR1}
	\lambda(L)\cn \langle \RR^{n+1}(L) \rangle+U_{X/S},
\end{equation}
where $U_{X/S}$ is a line bundle on $S$ that does not depend on the choice of $L$.

Observe that
\[
\RR^{n+1}(aL)=\sum_{i=0}^{n+1} a^i \frac{1}{i!}\Td^{n+1-i}(T_{X/S})c_1^i(L).
\]
So 
\[
\Td^{n+1-i}(T_{X/S})c_1^i(L)=
i!\sum_{j=0}^{n+1}A^{n+1}_{i,j}\RR^{n+1}(jL).
\]
See Appendix~\ref{sec:appCom} for the definition of $A^{n+1}$.

Therefore, we conclude that
\begin{lemma}\label{lma:reductiondrr}
(\ref{eq:DRR1}) is equivalent to the following statement: there is a functorial isomorphism
\begin{equation}\label{eq:reductiondrr}
\langle \Td^{n+1-i}(T_{X/S})c_1^i(L) \rangle_{X/S}\cn
i!\sum_{j=0}^{n+1}A^{n+1}_{i,j}\lambda_{X/S}(jL)
\end{equation}
for any $0< i\leq n+1$.
\end{lemma}

We observe the following lemma
\begin{lemma}\label{lma:temp1}
Assume $n>0$, $i>0$. Then (\ref{eq:reductiondrr}) holds if the following holds:
\begin{equation}\label{eq:furred}
\begin{split}
\langle \Td^{n+1-i}(T_{X/S})c_1(L_0)c_1^{i-1}(L)\rangle_{X/S}\cn &
(i-1)!\sum_{j=0}^{i-1}A^{i-1}_{i-1,j} \sum_{k=0}^{n+1}A^{n+1}_{i,k}\lambda(kL_0+kjL)\\
&-
\frac{i-1}{2} i! \sum_{k=0}^{n+1}A^{n+1}_{i,k}\lambda(kL)
\end{split}
\end{equation}
for any line bundle $L_0$ on $X$.
\end{lemma}
\begin{proof}
Recall that by partial polarization (Proposition~\ref{prop:partpol}), we have
\[
c_1(L_0)c_1^{i-1}(L)=\frac{1}{i}\sum_{j=0}^{i-1}A^{i-1}_{i-1,j}(c_1^{i}(L_0+j L))-\frac{i-1}{2}c_1^i(L),
\]
where $L_0$ is a line bundle on $X$.

We have
\begin{align*}
\langle \Td^{n+1-i}(T_{X/S})c_1(L_0)c_1^{i-1}(L)\rangle_{X/S}\cn &\frac{1}{i}\sum_{j=0}^{i-1}A^{i-1}_{i-1,j} \langle \Td^{n+1-i}(T_{X/S}) c_1^{i}(L_0+j L)\rangle_{X/S}\\
&-\frac{i-1}{2}\langle \Td^{n+1-i}(T_{X/S}) c_1^{i}(L)\rangle_{X/S}.
\end{align*}
We claim that (\ref{eq:reductiondrr}) is implied by (\ref{eq:furred}).

Actually, it suffices to show when $L_0=L$, the RHS of (\ref{eq:furred}) is naturally isomorphic to the RHS of (\ref{eq:reductiondrr}). According to Theorem~\ref{thm:lambdarelation}, this is further equivalent to 
\[
\sum_{j=0}^{i-1}A^{i-1}_{i-1,j}  \sum_{k=0}^{n+1}A^{n+1}_{i,k}\binom{k+kj}{\ell}=
 \frac{i(i+1)}{2}\sum_{k=0}^{n+1}A^{n+1}_{i,k}\binom{k}{\ell},
\]
for any $0\leq \ell\leq n+1$.

According to Proposition~\ref{prop:A}, LHS is equal to
\[
\sum_{j=0}^{i-1}\frac{1}{(i-1)!}\binom{i-1}{j}(-1)^{i-1-j}(j+1)^i B_{\ell,i}.
\]
And according to Proposition~\ref{prop:binom1}, this is equal to
\[
\frac{i(i+1)}{2}B_{\ell,i}.
\]
For the RHS, Proposition~\ref{prop:A} shows that it is equal to
\[
\frac{i(i+1)}{2}B_{\ell,i}.
\]
The equality follows.
\end{proof}

\begin{lemma}\label{lma:temp2}
For $n>0$, $i>0$, the RHS of (\ref{eq:furred}) is additive with respect to $L_0$.
\end{lemma}
\begin{proof}
Actually, let $L'$ be a sufficiently ample line bundle on $X$ such that $L_0+L'$ is sufficiently ample. Assume $D$ is a section of $L'$.
Then
\[
\lambda(kL'+kL_0+kjL)
=\sum_{i=1}^{k} \lambda_{D/S}(iL'+kL_0+kjL)+ \lambda(kL_0+kjL).
\]
And similarly
\[
\lambda(kL'+kjL)
=\sum_{i=1}^{k} \lambda_{D/S}(iL'+kjL)+ \lambda(kjL).
\]
Hence the additivity of \ref{eq:furred} is implied by the following claims.
 
Claim 1: 
\begin{equation}\label{eq:temp3}
\sum_{j=0}^{i-1}\frac{1}{(i-1)!}\binom{i-1}{j}(-1)^{i-1-j}\sum_{k=0}^{n+1} A^{n+1}_{i,k}
\sum_{\ell=1}^{k}\left( 
\lambda_{D/S}(\ell L'+kjL+kL_0)-
\lambda_{D/S}(\ell L'+kjL)
\right)=0.
\end{equation}

Claim 2:
\begin{equation}\label{eq:temp4}
\sum_{j=0}^{i-1}\frac{1}{(i-1)!}\binom{i-1}{j}(-1)^{i-1-j}\sum_{k=0}^{n+1} A^{n+1}_{i,k}\binom{kj}{r}=\frac{i(i-1)}{2}\sum_{k=0}^{n+1}A^{n+1}_{i,k}\binom{k}{r}
\end{equation}
for any $0\leq r\leq n+1$.

We first prove Claim 2. According to Proposition~\ref{prop:A}, 
\[
\sum_{k=0}^{n+1} A^{n+1}_{i,k}\binom{kj}{r}=j^i B_{r,i}.
\]
Hence according to Proposition~\ref{prop:binom1}, LHS of (\ref{eq:temp4}) is equal to
\[
\frac{i(i-1)}{2}B_{r,i}
\]
RHS of (\ref{eq:temp4}) is the same by Proposition~\ref{prop:binom1}. Claim 2 follows.

In order to prove Claim 1, we apply inductive hypothesis to $D$, we then find 
\[
\lambda_{D/S}(\ell L'+kjL+kL_0)
=
\sum_{a+b+c\leq n}\frac{\ell^a k^{b+c} j^b}{a!b!c!}\langle \Td^{n-a-b-c}(T_{D/S})c_1^a(L') c_1^b(L) c_1^c(L_0)  \rangle_{D/S}.
\]

It suffices to show that all terms with $c>0$ has coefficients $0$.

Using Proposition~\ref{prop:binom1}, we find that the coefficients is non-zero only when $b\geq i-1$. By Proposition~\ref{prop:powersum}, we know that the sum of $\ell^a$ with respect to $\ell$ is a polynomial of degree $a+1$ in $k$ without constant term. So the sum with respect to $k$ is non-zero only when $b+c\leq i-1$. Each term with $c>0$ has coefficient $0$. Claim 1 follows. 
\end{proof}

\begin{lemma}\label{lma:temp3}
When $D$ is a regular section of $L_0$, the RHS of (\ref{eq:furred}) can be written as 
\begin{equation}\label{eq:ffred}
\sum_{j=0}^{i-1}\binom{i-1}{j}(-1)^{i-1-j} \sum_{k=0}^{n+1}A^{n+1}_{i,k}\sum_{\ell=1}^{k} \lambda_D(\ell L_0+kjL).
\end{equation}
\end{lemma}
\begin{proof}
We write 
\[
\lambda(kL_0+kjL)=\sum_{\ell=1}^{k} \lambda_D(\ell L_0+kjL)+\lambda(kjL).
\]
Let $I$ be the part of RHS of (\ref{eq:furred}) that involves $\lambda_D$. Then
\[
I=\sum_{j=0}^{i-1}\binom{i-1}{j}(-1)^{i-1-j} \sum_{k=0}^{n+1}A^{n+1}_{i,k}\sum_{\ell=1}^{k} \lambda_D(\ell L_0+kjL).
\]
Let $II$ be the remaining term. Then
\[
II=\sum_{j=0}^{i-1}\binom{i-1}{j}(-1)^{i-1-j} \sum_{k=0}^{n+1}A^{n+1}_{i,k} \lambda(kjL)-\frac{i-1}{2}i! \sum_{k=0}^{n+1}A^{n+1}_{i,k}\lambda(kL)
\]
We claim that $II$ is canonically trivial. To prove this, by Theorem~\ref{thm:lambdarelation}, it suffices to show that for any $0\leq b\leq n+1$, we have
\begin{equation}\label{eq:temp1}
\sum_{j=0}^{i-1}\binom{i-1}{j}(-1)^{i-1-j} \sum_{k=0}^{n+1}A^{n+1}_{i,k} \binom{kj}{b}=\frac{i-1}{2}i! \sum_{k=0}^{n+1}A^{n+1}_{i,k}\binom{k}{b}.
\end{equation}
But we known that 
\[
\sum_{k=0}^{n+1}A^{n+1}_{i,k} \binom{kj}{b}=j^i B_{b,i}.
\]
Hence the LHS of (\ref{eq:temp1}) becomes
\[
\frac{i-1}{2}i! B_{b,i}
\]
using Proposition~\ref{prop:binom1}.

Similarly RHS of  (\ref{eq:temp1}) is equal to
\[
\frac{i-i}{2}i! B_{b,i}.
\]
\end{proof}
\begin{lemma}\label{lma:temp4}
The LHS of (\ref{eq:furred}) is equal to
\begin{equation}\label{eq:asum}
\sum_{a=0}^{n+1-i}\frac{(-1)^a}{a!}B_a \langle \Td^{n+1-i-a}(T_{D/S})c_1^a(L_0) c_1^{i-1}(L)\rangle_{D/S}
\end{equation}
and is equal to (\ref{eq:ffred}).
\end{lemma}
\begin{proof}
The expression of LHS of (\ref{eq:furred}) follows from the proof of Proposition~\ref{prop:RRrest}.

We apply inductive hypothesis to $D$, then we find (\ref{eq:ffred}) is equal to
\[
\sum_{j=0}^{i-1} \binom{i-1}{j}(-1)^{i-1-j}\sum_{k=0}^{n+1}A^{n+1}_{i,k}\sum_{\ell=1}^{k}\sum_{a=0}^n \sum_{b=0}^{n-a}\frac{1}{a! b!}\ell^a k^b j^b
\langle \Td^{n-a-b}(T_{D/S})c_1^a(L_0)c_1^b(L) \rangle_{D/S}.
\]

We rewrite the part without $U$ as
\begin{equation}\label{eq:temp2}
\sum_{a+b\leq n}\frac{1}{a!b!} \left(\sum_{j=0}^{i-1} \binom{i-1}{j}(-1)^{i-1-j}j^b\right)
\left(\sum_{k=0}^{n+1}A^{n+1}_{i,k} k^b \sum_{\ell=1}^{k}\ell^a  \right) \langle \Td^{n-a-b}(T_{D/S})c_1^a(L_0)c_1^b(L) \rangle_{D/S}.
\end{equation}

First observe that when $b<i-1$, the first bracket is zero by Proposition~\ref{prop:binom1}.

When $n\geq b> i$, the second bracket is clearly zero.
This is also true when $b=i$, as according to Proposition~\ref{prop:powersum}, $\sum_{\ell=0}^{k-1}\ell^a$ is a polynomial in $k$ of degree $a+1$ and the constant term is clearly $0$.

So (\ref{eq:temp2}) becomes
\[
\sum_{a=0}^{n-1+i} \frac{1}{a!}\left(\sum_{k=1}^{n}A^{n+1}_{i,k} k^{i-1} \sum_{\ell=0}^{k-1}\ell^a  \right) \langle \Td^{n-a-i+1}(T_{D/S})c_1^a(L_0)c_1^{i-1}(L) \rangle_{D/S}.
\]
According to Proposition~\ref{prop:powersum}, the bracket is equal to 
\[
\frac{1}{a+1}\sum_{j=0}^a \binom{a+1}{j}B_j (a+1-j)(-1)^{a-j}.
\]
So we finally reduce to prove
\[
(a+1)B_a=\sum_{j=0}^a \binom{a+1}{j}B_j (a+1-j)(-1)^{j}.
\]
For this identity, see Proposition~\ref{prop:bern}.
\end{proof}

So now, we have proved (\ref{eq:DRR1}).

The problem now is to find a trivialization of
\[
U_{X/S}=\lambda_{X/S}(0)-\langle \Td^{n+1}(T_{X/S}) \rangle_{X/S}.
\]

This has been done when $n=0$. Our proof in that case works for arbitrary $n$ without any change.

This finishes the proof of Theorem~\ref{thm:DRR}.

\subsection{Explicit formula for the Deligne-Riemann-Roch isomorphism}

Note that our proof is indeed constructive, so we have an explicit construction of the Deligne-Riemann-Roch isomorphism. We make it more explicit for further use.

Let $f,X,S$ be as in Theorem~\ref{thm:DRR}.

In order to avoid horrible notations, we make the following convention:
We regard the universal relations in Theorem~\ref{thm:lambdarelation} as equalities. Similar convention applies to the $\alpha$, $\beta$, $\gamma$, $\delta$ isomorphisms for Deligne pairings. We omit the notations for pull-back to universal extensions and for descending from the universal extension. Moreover, we use $Z(L)$ for the universal divisor corresponding to $L$.

Write 
\[
I_{X}(L):\lambda_{X/S}(L)\cn \langle \RR^{n+1}(L)  \rangle_{X/S}
\]
for the isomorphism established in (\ref{eq:DRR}).

\begin{proposition}\label{prop:explic}
    Let $i>0$, $n>0$, then for any line bundle $L$ on $X$, we have
	\[
	I_{X}(L)=\sum_{i=0}^{n+1}\frac{1}{i!}\sum_{j=0}^{i-1}\binom{i-1}{j}(-1)^{i-1-j}\sum_{k=0}^{n+1}A^{n+1}_{i,k}\sum_{\ell=1}^{k}  \left( I_{Z(L+L_0)}(\ell L_0+(kj+\ell )L)-I_{Z(L_0)}(\ell L_0+kjL)\right),
	\]
	where $L_0$ is a line bundle on $X$ such that $L_0$ and $L+L_0$ are both sufficiently ample.
\end{proposition}
This follows by examing the proof. 
\subsection{Compatiblity of Deligne-Riemann-Roch isomorphisms with exact sequences}\label{subsec:compati}
In this subsection, we prove that the Deligne-Riemann-Roch isomorphism (\ref{eq:DRR}) is compatible with short exact sequences.

To be more precise, let $f,X,S$ be as in Theorem~\ref{thm:DRR}, let 
\[
0\rightarrow F\rightarrow E\rightarrow G \rightarrow 0
\]
be an exact sequence of vector bundles on $X$. 

Then there is then a diagram
\[
\begin{tikzcd}
\lambda(E)\arrow[d,"\sim"]\arrow[r,"\sim"] &\langle \RR^{n+1}(E)\rangle \arrow[d,"\sim"]\\
\lambda(F)+\lambda(G)\arrow[r,"\sim"] &\langle \RR^{n+1}(F)\rangle+\langle \RR^{n+1}(G)\rangle
\end{tikzcd},
\]
where horizontal morphisms are the Deligne-Riemann-Roch isomorphisms.

We ask if there is a canonical isomorphism
\begin{equation}\label{eq:Iadd}
I_{X/S}(E)\cn I_{X/S}(F)+I_{X/S}(G),
\end{equation}
where $I_{X/S}(E)$ is the isomorphism in (\ref{eq:DRR}).

First consider the case where $X=\mathbb{P}V$ for some vector bundle $V$ on $S$. According to the explicit calculation made in Subsection~\ref{subsec:projsp}, (\ref{eq:Iadd}) can be easily constructed.

Then the construction in \cite{KM77} extends $I$ and (\ref{eq:Iadd}) to all perfect complexes.

Now in general, consider an embedding of $S$ schemes $i:X\hookrightarrow \mathbb{P}^N_S$. 
We may assume that all of $E$, $F$, $G$ admit a canonical complete flags by pulling back to the flag variety.
There is diagram:
\[
\begin{tikzcd}
I_{X/S}(F)\arrow[d,"\sim"]\arrow[r,dashrightarrow] & I_{X/S}(E)+ I_{X/S}(G)\arrow[d,"\sim"]\\
I_{\mathbb{P}^N_S/S}(i_*F)\arrow[r,"\sim"] &I_{\mathbb{P}^N_S/S}(i_*E)+I_{\mathbb{P}^N_S/S}(i_*G)
\end{tikzcd},
\]
where vertical isomorphisms are defined as follows: clearly it suffices to construct the vertical isomorphisms for line bundles. Let $L$ be a line bundle on $X$, the desired isomorphism is the canonical one induced by the canonical isomorphism
\[
\lambda_{\mathbb{P}^N_S/S}(i_*L)\cn \lambda_{X/S}(L).
\]

The desired compatibility isomorphism follows by completing the dotted morphism so that the diagram commutes.

We also note that the isomorphisms (\ref{eq:Iadd}) satisfy some obvious coherence conditions themselves.
\section{Refinement of Deligne-Riemann-Roch theorem}\label{sec:met}
\subsection{Definition of metrics on Deligne pairings}\label{subsec:metdel}
In this part, we shall define a metric on the Deligne pairings following \cite{Elk90}. Proofs will be omitted and we refer to \cite{Elk90} for proof.

Let $(f:X\rightarrow S)\in \mathcal{F}^n$. Assume furthermore that $S$, $X$ are complex quasi-projective varieties. Let $\hat{L}_i=(L_i,\varphi_i)$ be Hermitian line bundles on $X$. We shall define a metric on 
\[
\langle L_0,\dots,L_n\rangle_{X/S}.
\]
This line bundle with this particular metric will be denoted as
\[
\langle \hat{L}_0,\dots,\hat{L}_n\rangle_{X/S}.
\]
The metric is defined by induction on $n$. When $n=0$, we use the usual metric on the norm.

Now assume $n>0$. 
We first proceed in the case where $L_0$ is trivial, say $\hat{L}_0=(\mathcal{O}_X,\varphi)$. In this case, we require the following $\gamma$ isomorphism to be isometric
\[
\langle \mathcal{O}_X, L_1,\dots, L_n\rangle_{X/S}\cn
\mathcal{O}_S,
\] 
where the metric on the RHS is given by
\[
\int_{X/S}\varphi c_1(L_1)\wedge\cdots\wedge c_1(L_n).
\]
This is always continuous by a theorem of Stoll. See Appendix~\ref{sec:appFib}

By requiring the additivity isomorphism to be isometric, it suffices then to require the following $\beta$ isomorphism to be isometric:
\[
\langle \mathcal{O}_X(D), L_1,\dots, L_n\rangle_{X/S}\cn \langle  L_1,\dots, L_n\rangle_{D/S},
\] 
where $D$ is a non-zero effective relative Cartier divisor on $X$, the metric on $\mathcal{O}_X(D)$ is the canonical singular Hermitian metric. Notice that by the additivity property and the anomaly formula, putting a $L^1$-metric on $\mathcal{O}_X(D)$ makes sense.

These special cases then generate the metric on all Deligne pairings. 
We refer to \cite{Elk90} for the well-defineness.
The resulting metrics are in general only continuous. 
(See Appendix~\ref{sec:appFib})

Let $\hat{E}$ be a Hermitian vector bundle on $X$. The \emph{Segre forms} $s_i(\hat{E})$ of $\hat{E}$ can be defined using
\[
\sum_{i\geq 0}s_i(\hat{E})\sum_{i\geq 0}c_i(\hat{E})=1.
\]

An equivalent definition can be given by
\[
s_a(\hat{E})=\pi_*(c_1(\mathcal{O}_{\mathbb{P}E}(1))^{r+a}),
\]
where $\pi:\mathbb{P}E \rightarrow X$ and the canonical metric on $\mathcal{O}_{\mathbb{P}E}(1)$ is used. For the proof of the equivalence of the two definitions, we refer to \cite{Mou04} Proposition~6.

Now for the general definition of metrics on Deligne pairings. Let $\hat{E}_i$ ($i=1,\dots,m$) be Hermitian vector bundles on $X$. Let $P(c_j(E_i))$ be a homogeneous polynomial of degree $n+1$. 

Set $\mathbb{P}=\mathbb{P}(E_1)\times \cdots \times \mathbb{P}(E_m)$.

We define a metric on $\langle P(c_j(E_i)) \rangle_{X/S}$ by taking the linear extension of 
\[
\langle s_{k_1}(\hat{E}_1)\dots s_{k_m}(\hat{E}_m)\rangle_{X/S}:=\langle \mathcal{O}_{\mathbb{P}E_1}(1)\{r_1+k_1\},\dots,\mathcal{O}_{\mathbb{P}E_m}(1)\{r_m+k_m\} \rangle_{\mathbb{P}/S},
\] 
where $s_j$ denotes the $j$-th Segre class, $\sum_j k_j=n+1$, \{$a$\} denotes repeating $a$-times and canonical metrics on $\mathcal{O}(1)$ bundles are used.

We write the line bundle with this specific metric as
\[
\langle P(c_j(\hat{E}_i))\rangle_{X/S}.
\]

We list the basic properties of this metric. $Q$ will denote a homogeneous Chern polynomial of proper degrees.
\begin{theorem}
1.
\[
c_1(\langle P(c_j(\hat{E}_i))\rangle_{X/S})=\int_{X/S}P(c_j(E_i)).
\]
2. Let $\hat{L}_i$ be Hermitian line bundles, then
\[
\langle c_1(\hat{L}_0)\dots c_1(\hat{L}_n))\rangle_{X/S}=\langle \hat{L}_0,\dots,\hat{L}_n\rangle_{X/S}.
\]
3. Let $\mathcal{E}=(0\rightarrow E\rightarrow F\rightarrow G\rightarrow 0)$ be a short exact sequence. Let $P_k$ be $c_k$ or $\Td^k$. Let $\tilde{P}$ be the Bott-Chern secondary class corresponding to $P$, then there is a natural isometric isomorphism
\[
\langle P_k(\hat{F}) Q\rangle_{X/S}\cn \left\langle \sum_{i+j=k}P_i(\hat{E})P_j(\hat{G}) Q \right\rangle_{X/S}\otimes \left(\mathcal{O}_S,\int_{X/S}\tilde{P}_k(\mathcal{E})Q\right).
\]
\end{theorem}
For the proof of these facts and further properties of the metric, we refer to \cite{Elk90}.
\begin{remark}
Note that a priori, we only know that LHS in 1 is a distribution.
\end{remark}
\subsection{Deligne-Riemann-Roch theorem with metric}
Define the following series
\[
R(x):=\sum_{n\geq 1, n\text{ odd}} \left((1+\frac{1}{2}+\dots+\frac{1}{n})\zeta(-n)+2\zeta'(-n)\right)\frac{x^n}{n!}.
\]

Now we can prove a refined version of Theorem~\ref{thm:DRR}.

\begin{theorem}\label{thm:DRRm}
Let $X,S$ be complex quasi-projective manifolds. Let $f:X\rightarrow S$ be a projective smooth K\"{a}hler fibration of pure relative dimension $n$. Let $\omega$ be a relative K\"{a}hler metric on $T_{X/S}$. Write $\hat{T}=(T_{X/S},\omega)$.
Let $\hat{E}$ be a Hermitian vector bundle on $X$, then there is a functorial isometric isomorphism
\begin{equation}\label{eq:RRm}
\hat{\lambda}_{X/S}(\hat{E})\cn \langle\RR(\hat{E}) \rangle_{X/S}+\left(\mathcal{O}_S,\int_{X/S} \left[\RR(\hat{E})\mathrm{R}(\hat{T})\right]^n \right).
\end{equation}
\end{theorem}
As the proof is completely parallel to that in \cite{GS91}, we only give a sketch.

Here $\hat{\lambda}_{X/S}(\hat{E})$ denotes the determinant line bundle $\lambda_{X/S}(E)$ with the Quillen metric.

Before proceeding, we make a few simplifications.

1. The validity of \ref{eq:RRm} is independent of the metric on $E$. This is because both sides have the same anomaly formulas. (See~\cite{Sou94lec} Chapter~VI, Theorem~4)

2. We may always assume that $E$ is a direct sum of line bundles by passing to a flag variety as in the proof of Theorem~\ref{thm:DRR}. According to 1, we may assume further that $\hat{E}$ is a direct sum of Hermitian line bundles.
we reduce directly to the case of Hermitian line bundles.

We proceed in several steps. 

Step 1. Consider the case where $X$ is a projective space bundle over $S$, say $X=\mathbb{P}V$, where $V$ is a vector bundle of rank $n+1$ on $S$. 
 
We recall the following result:
\begin{theorem}[\cite{GS91}]
1. The line bundle $\lambda_{\mathbb{P}V/S}(\hat{0})$ is trivial. Under the trivialization induced by the constant section $1\in f_*(\mathbb{P}V)$, then
\[
\hat{\lambda}_{\mathbb{P}V/S}(\hat{0})=\left(\mathcal{O}_S,-\log (n!)-\sum_{q\geq 0}(-1)^{q+1}q\zeta_q'(0)\right).
\]

2. Then integral 
\[
-\int_{X/S}\left[\RR(\hat{0})\mathrm{R}(\hat{T})\right]^n=\left[(n+1)\left(\frac{x}{1-e^{-x}}\right)^{n+1}R(x)\right]_{n}.
\]
Here $[\cdot]^n$ denotes the homogeneous part of degree $n$ and $[\cdot]_n$ denotes the coefficient of $x^n$.

3. 
\[
\langle \RR^{n+1}(0)\rangle_{\mathbb{P}V/S}=\left(\mathcal{O}_S,\sum_{i=1}^n \sum_{j=1}^i \frac{1}{j} \left[
\left(\frac{x}{1-e^{-x}}\right)^{n+1}
\right]_{n+1}+\left[ \int_{0}^1 \frac{\varphi(t)-\varphi(0)}{t}\,\mathrm{d}t \right]_n \right),
\]  
where
\[
\varphi(t)=\left( \frac{1}{tx}-\frac{e^{-tx}}{1-e^{-tx}}\right)\left(\frac{x}{1-e^{-x}}\right)^{n+1}.
\]

4. Theorem~\ref{thm:DRRm} holds for this specific case.
\end{theorem}

Step 2. Consider a general $f:X\rightarrow S$, in this case, we can write
\[
\begin{tikzcd}
X \arrow[r,"i",hookrightarrow] \arrow[dr,"f"]& \mathbb{P}V \arrow[d,"\pi"] \\ 
&S
\end{tikzcd},
\]
where $V$ is a vector bundle of rank $d+1$ on $S$.

Now we need to calculate $\lambda(\hat{L})$, where $\hat{L}$ is a Hermitian line bundle on $X$.

Let 
\[
\mathbb{E}=(0\rightarrow E_N\rightarrow \cdots \rightarrow E_1)
\]
be a locally free resolution of $i_*L$. We take proper metrics on $E_i$ so that the Bismut condition (A) is satisfied. (See~\cite{Bis90})

We have the following result of Bismut-Lebeau(\cite{BL91})
\[
\begin{split}
\hat{\lambda}(\hat{L})-\sum_i (-1)^i \hat{\lambda}(\hat{E}_i)=
&\Big(\mathcal{O}_S,  -\int_{\mathbb{P}V/S}\Td(\hat{T}_{\mathbb{P}V/S})T(\hat{\mathbb{E}})+\int_{X/S}\Td^{-1}(\hat{N})\widehat{\Td}(\mathbb{E}) \ch(\hat{L})\\
&-\int_{\mathbb{P}V/S}\Td(\hat{T}_{\mathbb{P}V/S})\mathrm{R}(\hat{T}_{\mathbb{P}V/S})\ch(\hat{\mathbb{E}})+\int_{X/S}\Td(\hat{T}_{X/S})\mathrm{R}(\hat{T}_{X/S})\ch(\hat{L})\Big).
\end{split}
\]
We denote the four terms as $A,B,C,D$.

Here $T$ denotes the Bott-Chern current (see~\cite{BGS90}).

Now we have all necessary pieces of information for proving our theorem.

This proof goes as in \cite{GS92}, so we only sketch it.
\begin{proof}[Proof of Theorem~\ref{thm:DRRm}]
We calculate
\[
\sum_i (-1)^i \hat{\lambda}(\hat{E}_i)=
\sum_i (-1)^i \langle \RR(\hat{E}) \rangle_{\mathbb{P}V/S}+(\mathcal{O}_S,I),
\]
where 
\[
I=-\sum_{i}(-1)^i \int_{\mathbb{P}V/S}\left[ \RR(\hat{E}_i)\mathrm{R}(\hat{T}_{\mathbb{P}V/S})
\right]^d.
\]

We have by Grothendieck-Riemann-Roch
\[
\sum_i (-1)^i \langle \RR(\hat{E}_i) \rangle_{\mathbb{P}V/S}=\langle \RR^{n+1}(\hat{L})\rangle_{X/S}+(\mathcal{O}_S,A+B).
\]
As for $I$, we have
\[
I=C+D-\int_{X/S} \left[\RR(\hat{E})\mathrm{R}(\hat{T})\right]^n.
\]
The theorem follows.
\end{proof}

\section{Applications}
\label{sec:ap}
In this section, we give a few applications of the Deligne-Riemann-Roch theorem.
\subsection{Grothendieck-Riemann-Roch theorem}
\begin{theorem}[Grothendieck-Riemann-Roch Theorem]
Let $f:X\rightarrow S$ be a projective flat l.c.i. morphism between pure relative dimension $n$. Assume that $S$ is $\mathbb{Q}$-factorial. Let $E$ be a vector bundle on $X$, then
\[
c_1(\lambda_{X/S}(E))=\int_{X/S} \RR^{n+1}(E).
\]
\end{theorem}
\begin{proof}
It suffices to apply $c_1$ to (\ref{eq:DRR}).
\end{proof}
\subsection{Arithmetic Riemann-Roch theorem}
$f:X\rightarrow S$ be a projective flat morphism between arithmetic variety of pure relative dimension $n$. Assume that $f$ is smooth at infinity.
 
Let $\hat{E}_0,\dots, \hat{E}_m$ be Hermitian vector bundles on $X$. Let $P(c_i(\hat{E}_j))$ be a homogeneous polynomial of degree $n+1$.

We can form the Deligne pairing
\[
\langle P(c_i(E_j)) \rangle_{X/S}.
\]
We can define a Hermitian metric on this line bundle by using Quillen metric on the infinite fibre of $S$. We notice that it follows from the definition of Quillen metric that the result metric is invariant under complex conjugation.

We shall write the resulting Hermitian line bundle on $S$ as
\[
\langle P(c_i(\hat{E}_j)) \rangle_{X/S}.
\] 
\begin{theorem}\label{thm:aric1} Assumptions as above,
\[
\hat{c}_1\langle  P(c_i(\hat{E}_j)) \rangle_{X/S}=f_*P(\hat{c}_i(\hat{E}_j)).
\]
\end{theorem}
\begin{proof}
We first deal with the case where $P$ is a product of $c_1$ of line bundles. Let $\hat{L}_0,\dots,\hat{L}_n$ be Hermitian line bundles on $X$. We write
\[
\langle \hat{L}_0,\dots,\hat{L}_n\rangle_{X/S}:=\langle c_1(\hat{L}_0),\dots,c_1(\hat{L}_n)\rangle_{X/S}.
\]
In this case, according to the additivity of both sides, we may assume further that $L_0$ is sufficiently ample and there is a non-zero effective relative Cartier divisor $D$ on $X$ such that $L_0=\mathcal{O}_X(D)$. 
Furthermore, the validity of this result is clearly independent of the choice of metric on $L_0$, so we may assume that $\hat{L}_0=\hat{\mathcal{O}}_{X}(D)$. 
Recall the $\beta$ isomorphism, namely
\[
\langle \hat{L}_0,\dots,\hat{L}_n\rangle_{X/S}\cn \langle \hat{L}_1,\dots,\hat{L}_n\rangle_{D/S}.
\]
It is clear that the $\hat{c}_1$ class of both sides are equal. So by induction, we may assume further that $n=0$. In this case, we need to prove
\[
\hat{c}_1(N_{X/S}(\hat{L}))=f_* \hat{c}_1(\hat{L}),
\]
which follows immediately by definition.

Now the general case can be proved exactly as in \cite{Elk90} since $\hat{c}_1$ and $c_1$ have the same functorial properties.
\end{proof}
\begin{corollary}[Arithmetic Riemann-Roch Theorem]\label{cor:arr}
Let $f:X\rightarrow S$ be as above. Let $\hat{E}$ be a Hermitian vector bundle on $X$. Then
\[
\hat{c}_1(\hat{\lambda}_{X/S}(\hat{E}))=f_*\left(
\RR^{n+1}(\hat{E}) \right)+a\left(f_*\left[\RR(\hat{E})\mathrm{R}(\hat{T})\right]^n\right),
\]
where $a:H^0(X)\rightarrow \widehat{\CH}^1(X)$ is defined as in \cite{Sou94lec} Page~56.
\end{corollary}
\begin{proof}
It suffices to apply $\hat{c}_1$ to (\ref{eq:RRm}) and conclude by Theorem~\ref{thm:aric1}. 
\end{proof}

\subsection{Analytic torsion}
We now apply our theorem to study the analytic torsion. 
We use the following assumptions throughout this subsection:

Let $X$ be a projective manifold of dimension $n$. Let $\omega$ be a fixed K\"{a}hler form on $X$. Let $\hat{E}$ be a Hermitian vector bundle of rank $r$ on $X$.

We are going to study the analytic torsion of $E$.

For this purpose, we write $\varphi_{L^2}(\hat{E})$ for the $L^2$ metric on $\lambda(E)$. Let $T(\hat{E})$ be the Ray-Singer torsion of $E$. It is well-known that
\begin{equation}\label{eq:Qui}
\varphi_Q(\hat{E})=\varphi_{L^2}(\hat{E})-T(\hat{E}),
\end{equation}
where $\varphi_Q(\hat{E})$ is the Quillen metric of $\hat{E}$.

According to Theorem~\ref{thm:DRRm}, we find
\begin{theorem} Assumptions as above,
\begin{equation}\label{eq:fulltorsion}
(\mathbb{C},T(\hat{E}))=(\lambda(E),\varphi_{L^2}(\hat{E}))-\langle\RR^{n+1}(\hat{E})\rangle_{X}+(\mathbb{C},\int_{X} \left[\RR(\hat{E})\mathrm{R}(\hat{T})\right]^n).
\end{equation}
\end{theorem}
That is, we have a full expression of the analytic torsion.

Similar results are proved in \cite{Fin17} Theorem~1.3 and \cite{BV89}.

We still have to calculate RHS explicitly in order to compare these results.

\newpage
\appendix
\section{Combinatorial identities (by Wu Baojun)}\label{sec:appCom}
We shall use the following conventions $0^0=1$,
$\binom{0}{0}=1$,
the first Bernoulli number $B_1=1/2$.

Fix $n\in \mathbb{Z}_{\geq 0}$. Define the Vandermonde matrix of order $n$ as a matrix of order $(n+1)\times (n+1)$, subindices indexed by $0,\dots,n$:
\[
V^n_{i,j}:=(i^j).
\]

The following facts are well-known
\begin{proposition}\label{prop:vand}
\begin{enumerate}
\item
We have
\[
\det V^n=\prod_{0\leq i<j\leq n}(j-i).
\]
\item
Let $A^n$ be the inverse matrix of $V^n$, then
\[
A^n_{n,j}=\frac{1}{n!}\binom{n}{j}(-1)^{n-j}.
\]
\end{enumerate}
\end{proposition}
We observe that
\[
A^1=
\begin{pmatrix}
1 & 0\\
-1 & 1
\end{pmatrix}.
\]
This is need for some computations.

We study a few properties of $A^n$ matrices.
\begin{proposition}\label{prop:A}
Let $f(x)=\sum_r a_r x^r$ be a polynomial of degree at most $n$. Then
\[
\sum_{i=0}^n A^{n}_{j,k}f(k)=a_j.
\]

In particular, when $f(k)=\binom{ak}{b}$, $0\leq b\leq n$, we have
\[
\sum_{k=0}^n A^{n}_{j,k}\binom{ak}{b}=a^j B_{b,j},
\]
where
\[
B_{b,j}:=\frac{(-1)^{b-j}}{b!}[(x-1)\dots(x-(b-1))]_{j-1}
\]
\end{proposition}

\begin{proposition}\label{prop:binom1}
For any $n\in \mathbb{Z}_{> 0}$
\[
\frac{1}{n!}\sum_{j=0}^n \binom{n}{j}(-1)^{n-j}j^{k}=
\left\{
\begin{aligned}
&\frac{n(n+1)}{2} , &\text{ if } k=n+1,\\
&1 , &\text{ if } k=n,\\
&0 , &\text{ if } 0\leq k< n.
\end{aligned}
\right.
\]
\end{proposition}
\begin{proof}
When $0\leq k\leq n$, this follows readily from Proposition~\ref{prop:vand}.

When $k=n+1$, by taking inverse of Vandermonde matrix, we find  that the sum is the coefficient $C_n$ in the following expansion
\[
a^{n+1}=\sum_{i=0}^n C_i a^i,
\]
where $a=0,\dots,n$. Equivalently,  $C_n$ is the coefficient of $x^n$ in 
\[
x(x+1)\dots(x+n)
\]
Hence 
\[
C_n=\frac{n(n+1)}{2}.
\]
\end{proof}

Sometimes, we need a partial polarization of a polynomial of the form $y^{n+1}$ to get $xy^{n}$. We prove
\begin{proposition}\label{prop:partpol}
In the ring $\mathbb{Q}[x,y]$, we have
\[
xy^{n}= \frac{1}{n+1}\sum_{j=0}^n A^{n}_{n,j}(x+jy)^{n+1}  - \frac{n}{2}y^{n+1}.
\]
\end{proposition}
\begin{proof}
Indeed, we have the following relation
\[
(x+jy)^{n+1}-j^{n+1}y^{n+1}=\sum_{i=0}^n j^i \binom{n+1}{i}x^{n+1-i}y^i.
\]
Take $j=0,\dots,n$. According to Proposition~\ref{prop:vand}, we have
\[
(n+1)xy^n=\sum_{j=0}^n A^n_{n,j}\left((x+jy)^{n+1}-j^{n+1}y^{n+1}\right).
\]
The result follows from \ref{prop:binom1}.
\end{proof}
\begin{proposition}\label{prop:powersum}
Let $p,n\in \mathbb{Z}_{\geq 0}$, then
\[
\sum_{k=1}^n k^p=\frac{1}{p+1}\sum_{j=0}^p \binom{p+1}{j}B_j n^{p+1-j}.
\]
\end{proposition}

We have the following property of Bernoulli numbers.
\begin{proposition}\label{prop:bern}
Let $m,k\in \mathbb{Z}_{\geq 0}$, then
\begin{align*}
\sum_{k=0}^m \binom{m+1}{k}B_k &= m+1,\\
\sum_{k=0}^m (-1)^k\binom{m+1}{k}B_k &= \delta_{m,0},\\
1-\sum_{k=0}^{m-1}\binom{m}{k}\frac{B_k}{m-k+1}&=B_m.
\end{align*}
where $\delta_{m,0}$ is the Kronecker delta.

In particular,
\[
B_m=\frac{1}{m+1}\sum_{j=0}^m \binom{a+1}{j}B_j (m+1-j)(-1)^{j}.
\]
\end{proposition}
\section{Category theory}\label{sec:appPic}

Recall that a \emph{Grothendieck universe} is a strongly inaccessible cardinal $\kappa$. We fix one Grothendieck universe and assume that all sets, categories under consideration are $\kappa$-small. For this concept, we refer to \cite{Lur09} Appendix~A.1.

\subsection{Strictification of a commutative functor}
\label{subsec:strict}
Let $\mathbf{C}$, $\mathbf{D}$ be two ($\kappa$-small) groupoids. 

Assume that there is a functor
\[
F:\mathbf{C}^2\rightarrow \mathbf{D}.
\]
Let $\tau$ be the automorphism of $\mathbf{C}^2$ defined by exchanging the two components.

Assume that there is a natural isomorphism
\[
G:F\cn F \circ \tau.
\]

We can keep to notation $G$ everywhere, but we prefer to drop it in the following manner: Consider the set 
\[
A=\{(a,b):a,b\in \pi_0\mathbf{C}\}/\sim,
\]
where $(a,b)\sim (c,d)$ if{f} ($a=d$ and $b=c$) or ($a=c$ and $b=d$) and $\pi_0\mathbf{C}$ is the set of isomorphism classes in $\mathbf{C}$.

We choose a representative $(\alpha,\beta)$ for each class in $A$, where $\alpha,\beta\in \mathbf{C}$. And define $F'(\beta,\alpha)=F'(\alpha,\beta)=F(\alpha,\beta)$. Now for a general pair $(\gamma,\delta)$ of objects in $\mathbf{C}$, there is a unique representative $(\alpha,\beta)$ such that there is an arrow $i:(\gamma,\delta)\rightarrow (\alpha,\beta)$ or $(\gamma,\delta)\rightarrow (\beta,\alpha)$.

In the first case, set 
\[
F'(\gamma,\delta)=F(\gamma,\delta).
\]

In the second case, set
\[
F'(\gamma,\delta)=F(\delta,\gamma).
\]

In any case, we have defined a strictly commutative functor $F'$.
Replacing $F$ by $F'$, we may always consider strictly commutative functors instead of commutative functors with associated symmetric data. (See~\cite{Duc05} for example) Actually, there is no essential difference between these two point of views.

This construction obviously generalizes to more variables. 

In particular, we can safely assume that tensor products in a Picard category is strictly commutative.
In the same way, we may assume that Deligne pairings are strictly commutative.
In any case, there are some obvious compatiblity conditions to verify, but these are always sort of trivial.

\subsection{Picard categories}
\begin{definition}
A \emph{Picard category} is a symmetric monoidal groupoid $(\mathbf{C},\otimes)$ such that 
for each $X\in \mathbf{C}$, $\cdot \mapsto X \otimes \cdot$ is an autoequivalence of $\mathbf{C}$. 
\end{definition}

In this appendix, all Picard categories will be assumed to be strictly commutative and finite direct limits exist. See \cite{Del87} for the precise meaning. 

For the basic concepts concerning monoidal categories, see~\cite{Lur09} Appendix~A.1.3.

Throughout this subsection, $\mathbf{C}$, $\mathbf{D}$ will denote Picard categories.

\begin{definition}
An \emph{additive morphism} $\mathbf{C}\rightarrow \mathbf{D}$ is a monoidal morphism between the underlying monoidal categories. More precisely, it is a
pair: $(F,\mu)$, where
\begin{itemize}
\item $F:\mathbf{C}\rightarrow \mathbf{D}$ is a functor.
\item $\mu_{X,Y}:F(X)\tp F(Y)\rightarrow F(X\tp Y)$ for each $X,Y\in \mathbf{C}$ is an isomorphism functorial in $X,Y$ and is associative for a triple $X,Y,Z$ in the obvious sense.
\end{itemize}

By abuse of language, we always say $F$ is an additive morphism between Picard categories and say $\mu$ is the additive datum of $F$. And we write
\[
F\in L^1(\mathbf{C},\mathbf{D}).
\]
\end{definition}

\begin{definition}\label{def:natiso}
Let $(F,\mu)$, $(F',\mu')$ be two additive morphisms $\mathbf{C}\rightarrow \mathbf{D}$, an \emph{additive natural isomorphism} $A:(F,\mu)\rightarrow (F',\mu')$ is a monoidal natural isomorphism transforming $\mu$ to $\mu'$.
\end{definition}

Warning: When talking about morphisms between general categories, we still have the concept of a natural isomorphism.

We shall use interchangeably the adjectives natural, functorial and canonical.

\begin{definition}
Let $F:\mathbf{C}^n\rightarrow \mathbf{D}$ be a functor. We say $F$ is \emph{multi-additive} if when fixing $n-1$ components, $F$ restricts to a morphism between Picard categories $\mathbf{C}\rightarrow \mathbf{D}$ (with $\mu$ given a priori), and such that the following axiom holds: let $a_1,a_2,b_1,b_2\in \mathbf{C}$, fixing other components, we may assume $n=2$, we require the two natural way of achieving 
\begin{equation}\label{eq:comp}
F(a_1,b_1)\otimes F(a_1,b_2)\otimes F(a_2,b_1) \otimes F(a_2,b_2)\rightarrow F(a_1\otimes a_2,b_1\otimes b_2)
\end{equation}
are equal.
\end{definition}
\begin{definition}
The functor $F$ is \emph{strictly symmetric} if for any $c_i\in \mathbf{C}$, $\sigma\in \mathfrak{S}_n$, we have
\[
F(c_1,\dots,c_n)=F(c_{\sigma 1}, \dots ,c_{\sigma n}).
\]
\end{definition}

Here we have used a strict $=$ in the definition. In \cite{Duc05}, Ducrot used the concept of symmetric datum as follows:
\begin{definition}
A \emph{symmetric datum} for $F$ is a system of natural isomorphisms $\gamma_{\sigma}:F\rightarrow F\circ \sigma$ for any $\sigma\in \mathfrak{S}_n$, the latter $\sigma$ acts on the subindices of the components of $\mathbf{C}$, such that
\[
\gamma_{\psi\circ \sigma}=\gamma_{\psi}\circ \gamma_{\sigma}
\]
for any $\psi,\sigma \in \mathfrak{S}_n$.
\end{definition}
\begin{definition}
A multi-additive functor $F$ together with a symmetric datum $\gamma$ is said to be \emph{symmetric} if $\gamma$ commutes with the additive data.
We write 
\[
F\in L^n(\mathbf{C},\mathbf{D})
\]
for this.
\end{definition}
We do not distinguish strictly symmetric $F$ and symmetric $F$ in this paper thanks to the method of Subsection~\ref{subsec:strict}.

\begin{definition}\label{def:multni}
Let $F,F'\in L^n(\mathbf{C},\mathbf{D})$, a \emph{functorial isomorphism} $A:F\rightarrow F'$ is a natural isomorphism of the underlying functors such that restricting to each component, $A$ becomes a natural isomorphism in the sense of Definition~\ref{def:natiso}.  
\end{definition}

All kinds of natural isomorphisms will be denoted by $\cn$ throughout this paper.
\subsection{Examples}
Given an exact category $\mathbf{T}$, Deligne constructed in \cite{Del87} a Picard category $V\mathbf{T}$ of virtual objects associated to $\mathbf{T}$. We refer to \cite{Eri12} for details.

Given a scheme $X$ or a complex analytic space, we define the following:

1. $\PIC(X)$ is the Picard category of holomorphic line bundles on $X$. 

2. $\Vir(X)$ is the Picard category of virtual vector bundles on $X$.
$\Vect(X)$ is the category of vector bundles on $X$.
So $\Vir(X)=V(\Vect(X))$.

3. $\PIC_{\mathbb{Q}}(X)$ is the Picard category of $\mathbb{Q}$-holomorphic line bundles on $X$. Namely, objects of $\PIC_{\mathbb{Q}}(X)$ are given by a pair $cL:=(L,c)$, where $L$ is a line bundle and $c\in \mathbb{Q}$. A morphism from $cL$ to $c'L'$ is given by an integer $N$, such that $cN,c'N\in \mathbb{Z}$ and an isomorphism $NcL \cn Nc'L'$. 

There is a functor $\PIC(X)\rightarrow \PIC_{\mathbb{Q}}(X)$ sending a line bundle $L$ to $1L$. 

Given a quasi-proejctive complex variety or a complex analytic space, we define furthermore

1. $\PICm(X)$ is the groupoid of smooth Hermitian line bundles on $X$.

2. $\PICm_0(X)$ is the Picard category of \emph{continuous} Hermitian line bundles on $X$.

\begin{remark}
The reason that we use continuous metric is that they are well-behaved under fibre integration. See \cite{Sto66} or Appendix~\ref{sec:appFib}.
\end{remark}

\section{Fibre integration}\label{sec:appFib}
In this part, we shall give a proof of the regularity of fibre integration following \cite{Sto16}.
\subsection{Functoriality}
Let $f:X\rightarrow S$ be a morphism between (connected) complex manifolds of pure relative dimension $n$. 
\begin{remark}
Let $f:X\rightarrow Y$ be a holomorphic map between connected complex manifolds, the following are equivalent

1. $f$ has pure relative dimension.

2. $f$ is open.

3. $f$ is flat. (In the sense of algebraic geometry)

See \cite{Fis06}, P.158 for a proof.
\end{remark}
There are two kinds of fibre integrations:

1. When $f$ is proper, we may define
\[
f_*:\mathcal{D}'^{p,q}(X)\rightarrow \mathcal{D}'^{p-n,q-n}(S).
\]
Let $\varphi \in \mathcal{D}'^{p,q}(X)$,
take $\psi \in \mathcal{D}_{p-n,q-n}(S)$, then $f^*\psi\in \mathcal{D}_{p,q}(X)$ as $f$ is proper, we thus define
\[
\int_{S}f_*\varphi\wedge \psi= \int_{X}\varphi\wedge f^*\psi.
\] 
It is not hard to check that this defines a current.

2. For general $f$, we have
\[
f_*:\mathcal{D}'^{p,q}_c(X)\rightarrow \mathcal{D}_c'^{p-n,q-n}(S),
\]
where the subindex $c$ denotes compactly supported ones. The definition is similar to the previous one.

In both cases, we write
\[
\int_{X/S}\varphi:=f_*\varphi.
\]

Clearly $(fg)_*=f_*g_*$ where $g:S\rightarrow S'$ is a flat proper morphism.

Observe that $\varphi\geq 0$ implies that $f_*\varphi\geq 0$.

Pull-back of current is in general not defined, but for currents with $C^\infty$(resp. $C^0$, $L^p$, $W^{k,p}$) coefficients, pull-back is defined and is in the same class.

However, pull-back to open sets is always defined. In particular, currents form a fine sheaf with respect to the analytic topology.

The following facts follow from definition.
\begin{proposition}
Let $g:S'\rightarrow S$ be a holomorphic map from a complex manifold $S'$ to $S$, let the corresponding Cartesian square be
\[
\begin{tikzcd}
X \arrow[r,"g'"]\arrow[d,"f'"] & X \arrow[d,"f"]\\
S'\arrow[r,"g"] & S
\end{tikzcd}
\]
Then we have the following identity
\begin{equation}
f'_* g'^*=g^*f'_*
\end{equation}
whenever both sides make sense.
\end{proposition} 
\begin{proposition}
Fibre integration commutes with differentiation and interior product. In particular, fibre integration induces a map on the cohomology level.
\end{proposition}

\begin{proposition}[Fubini]
Let $A$ and $B$ be complex manifolds, let $\varphi$ be a $L^1$ current of top degree on $A\times B$, then
\[
\int_{A\times B}\varphi=\int_A \int_{A\times B/A} \varphi.
\]
\end{proposition}
\begin{proposition}[Projection formula]
Let $f:X\rightarrow S$ be a proper holomorphic map between complex manifolds of pure relative dimension $n$, let $\varphi\in \mathcal{D}^{p,p}(S)$, let $\psi\in \mathcal{D}'_{p,p}(X)$, then
\[
f_*(\psi(f^*\varphi))= f_*\psi(\varphi).
\]
\end{proposition}
A suitable modification holds for non-proper case.
\subsection{Continuity of fibre integrals}
In this part, $\Delta^m$ will denote the open unit polydisk in $\mathbb{C}^m$.

\begin{theorem}[Stoll]\label{thm:Stoll}
Let $f:X\rightarrow S$ be a proper morphism between complex manifolds of relative dimension $n$. Let $\varphi$ be a smooth $(p,q)$-form on $X$, then $f_*\varphi$ has continuous coefficients.
\end{theorem}
We essentially follow the proof in \cite{Sto66} but make some simplifications.
\begin{proof}
We make the following observations:

1. It suffices to prove the theorem for $(n,n)$-forms. For $p<n$ or $q<n$, the result is trivial. Use the fact that interior product commutes with fibre integration, the result follows by induction.

2. The result is clearly local in both $X$ and $S$. So way may assume that $S=\Delta^m$, $X=\Delta^n\times \Delta^m$.

3. It suffices to prove the theorem for $(n,n)$-forms of pure type, namely combinations of $\mathrm{d}z_I\wedge \mathrm{d}\bar{z}_I$. General terms can be controlled using Cauchy-Schwarz inequality.

More precisely, let $g$ be a continuous function on $\Delta^n\times \Delta^m$, then we have
\begin{equation}\label{eq:CS}
|\int_{f^{-1}s} g\mathrm{d}z_I\wedge\mathrm{d}\bar{z}_J|^2\leq \int_{f^{-1}s}|g|\mathrm{d}z_I\wedge\mathrm{d}\bar{z}_I \cdot \int_{f^{-1}s}|g|\mathrm{d}z_J\wedge\mathrm{d}\bar{z}_J.
\end{equation}

Let $K$ be the only index that does not involve the coordinates of $\Delta^m$. Let $\eta$ be 
\[
\eta=\sum_{I\neq K} \mathrm{d}z_I\wedge\mathrm{d}\bar{z}_I
\]

Assume that $f^{-1}(0)$ is given by the vanishing of all coordinates corresponding to $\Delta^m$,
it is easy to see that
\[
\int_{X/S}|g|\eta
\]
is a continuous and is $0$ at the origin. This dominates all other terms in (\ref{eq:CS}) but the one with $I=J=K$, whose continuity follows again from the special case.

The technical assumption that $f^{-1}(0)$ be given by a hyperplane can be achieved outside the singular set of $f^{-1}(0)$, but recall that singular set is always of complex codimension at least $2$ hence of null measure, so they do not actually affect the integrals, we may come over them by the usual approximation tricks.

4. We may actually assume that $n=0$.

In fact, the proof reduces immediately to show the joint continuity of $\pi_{2*} g|_{f^{-1}(b)}(w)$ with respect to $(b,w)$. By taking a product by $\Delta^p$ of everything, we reduce immediately to the case of relative dimension $0$.

5. When $n=0$, the fibre integral is in fact the summation of fibre values counting multiplicity, to see this, we have the function thus defined, say $h$ is continuous on $S$ according to Lemma~\ref{lma:reduction1}. Hence defines a distribution still denoted by $h$. We prove that $h=f_*\varphi$, it suffices to take $\psi\in \mathcal{D}_{0,0}(S)$ and we need to show
\begin{equation}\label{eq:equality}
\int_S h\psi=\int_X \varphi f^*\psi.
\end{equation}
Actually it suffices to prove this result locally, so we may assume that $f:X\rightarrow S$ is a ramified covering, using the continuity of $\varphi$ and $h$, we may assume further that $f:X\rightarrow S$ is a covering without ramification, hence locally $X$ is disjoint copies of $S$, and $f$ maps each copy to $S$, the result becomes trivial now.
\end{proof}

\begin{lemma}\label{lma:reductiona}
Let $f:\Delta^p\times \Delta^q\rightarrow \Delta^p$ be a holomorphic map. Let $g:\Delta^p\times \Delta^q \rightarrow \mathbb{C}$ be a $L^1$ function with compact support(distributional sense). Let $\Omega$ be a smooth volume form on $\Delta^q$. Let $\pi_2:\Delta^p\times \Delta^q\rightarrow \Delta^q$ be the projection onto the second factor. Then for each $b\in \Delta^p$, we have
\begin{equation}
f_*(g\pi_2^*\Omega)(b)=\int_{\Delta^q} \left(\pi_{2*} g|_{f^{-1}(b)}\right)\Omega.
\end{equation}
\end{lemma}
This is an easy consequence of the projection formula and the Fubini theorem, we leave the proof to readers.

\begin{lemma}\label{lma:reduction1}
Let $f:X\rightarrow S$ be a proper holomorphic map between complex manifolds of pure relative dimension $0$, let $g$ be a continuous function on $X$, then $h(s):=\sum_{a\in f^{-1}(s)} \nu(a) g(a)$ is a continuous function of $s\in S$.
\end{lemma}
Here $\nu(a)$ is the multiplicity of $a$ in the fibre.
\begin{proof}
We reduce to local model as before, let $a_1,\dots,a_r$ be different points in $f^{-1}(0)$. 

Take neighbourhoods $U_i$ of $a_i$ such that 

1. $\bar{U}_i$ are pairwise disjoint

2. $g$ varies by no more than $\epsilon>0$ on each $U_i$.

Let $V$ be a neighbourhood of $0$ in $S$, so that 

1. $f^{-1}(V)\cap \partial U_i=\emptyset$.

2. $f^{-1}(V)\cap (\supp g-\cup_i U_i)=\emptyset$. 

Let $h=f_* g$. For $z\in X$, $s\in S$, define
$\nu(z,s)$ to be $0$ when $f(z)\neq s$, and to be the multiplicity of $f$ at $z$ when $f(z)=s$.

For $w\in V$, we find
\[
|h(w)-h(s)|\leq |\sum_i \sum_{z\in U_i}\nu(z,w)(g(z)-g(a_i))|\leq C\epsilon.
\]
Here we have applied the Stoll-Rouché theorem(\cite{Sto66B}) that the total multiplicity over a point in $S$ is locally constant.
\end{proof}

\begin{theorem}
Let $f$ be a proper smooth morphism of relative dimension $n$ between complex manifolds, let $\varphi$ be a smooth $(p,q)$-form on $X$, then $f_*\varphi$ is a smooth form on $S$.
\end{theorem}
Here \emph{smooth} means smooth in the algebraic geometric sense, namely, all fibres are smooth and $f$ is flat.
\begin{proof}
We reduce as before to $(n,n)$-forms. By Theorem~\ref{thm:Stoll}, $f_*\varphi$ is a continuous function on $S$.

Define a function $h$ on $S$ by
\[
h(s)=\int_{f^{-1}(s)}\varphi|_{f^{-1}(s)}.
\]
It follows from Leibnitz theorem that $h$ is a smooth function. 

We claim that $h=f_*\varphi$. The proof is similar to the proof of (\ref{eq:equality}) and we leave it to readers. 
\end{proof}
\begin{corollary}
Let $f:X\rightarrow S$ be a proper morphism between complex manifolds of pure relative dimension, let $\varphi$ be a smooth $(p,q)$-form on $X$, then $f_*\varphi$ is smooth on a dense open subset of $S$ and continuous on all $S$.
\end{corollary}

Finally we remark that all these theorems continue to hold without the properness assumption, but we have to concentrate on compactly supported currents. Some of the results generalize to normal complex spaces.

\begin{definition}
Let $X$ be a compact Kähler manifold, let $\varphi\in \mathcal{D}'^{1,1}(X)$, we say $\varphi$ \emph{has $C^\alpha$ primitive} if for some(hence any) smooth representative $\psi$ of the cohomology class of $\varphi$, there is a $C^\alpha$ function $f$ such that
\[
\varphi=\psi+\ddc f.
\]
\end{definition}
The following theorem is essentially due to Tian. (\cite{Tia97})
\begin{corollary}\label{thm:calphapri}
Let $f:X\rightarrow S$ be a proper morphism between closed Kähler manifolds of pure relative dimension $n$. Let $\varphi \in \mathcal{D}^{n+1,n+1}(X)$. Then $f_*\varphi$ has $C^{1,1/2}$ primitive.
\end{corollary}
\begin{proof}
This follows from Theorem~\ref{thm:Stoll} and standard elliptic estimates.
\end{proof}

\bibliography{cscK}

\bigskip
\begin{small}
Mingchen~Xia, 
\textsc{Department of Mathematics, \'{E}cole Normale Sup\'{e}rieure, Paris, France 75005}

\par\nopagebreak
  \textit{E-mail address}: \texttt{mcxia@clipper.ens.fr}
\end{small}

\end{document}